\documentclass[11pt,reqno]{amsart}
\usepackage{amsmath,amssymb,mathrsfs}
\usepackage{graphicx,cite}
\usepackage{color}
\usepackage{subfigure}

\newtheorem{theorem}{Theorem}[section]
\newtheorem{corollary}[theorem]{Corollary}
\newtheorem{lemma}[theorem]{Lemma}

\newtheorem{remark}[theorem]{Remark}

\newtheorem{assumption}{Assumption}

\numberwithin{equation}{section}

\allowdisplaybreaks[4]

\begin{document}

\title[Inverse scattering for the biharmonic wave equation]{Inverse scattering for the biharmonic wave equation with a random potential}

\author{Peijun Li}
\address{Department of Mathematics, Purdue University, West Lafayette, Indiana 47907, USA}
\email{lipeijun@math.purdue.edu}

\author{Xu Wang}
\address{LSEC, ICMSEC, Academy of Mathematics and Systems Science, Chinese Academy of Sciences, Beijing 100190, China, and School of Mathematical Sciences, University of Chinese Academy of Sciences, Beijing 100049, China}
\email{wangxu@lsec.cc.ac.cn}

\thanks{The first author is supported in part by NSF grants DMS-1912704 and DMS-2208256. The second author is supported by NNSF of China (11971470 and 11871068).}



\begin{abstract}
We consider the inverse random potential scattering problem for the two- and three-dimensional biharmonic wave equation in lossy media. The potential is assumed to be a microlocally isotropic Gaussian rough field. The main contributions of the work are twofold. First, the unique continuation principle is proved for the fourth order biharmonic wave equation with rough potentials and the well-posedness of the direct scattering problem is established in the distribution sense. Second, the correlation strength of the random potential is shown to be uniquely determined by the high frequency limit of the second moment of the scattered field averaged over the frequency band. Moreover, we demonstrate that the expectation in the data can be removed and the data of a single realization is sufficient for the uniqueness of the inverse problem with probability one when the medium is lossless. 
\end{abstract}

\maketitle

\section{Introduction}

Scattering problems arise from the interaction between waves and media. They play a fundamental role in many scientific areas such as medical imaging, exploration geophysics, and remote sensing. Driven by significant applications, scattering problems have been extensively studied by many researchers, especially for acoustic and electromagnetic waves \cite{CK13, N00}. Recently, scattering problems for biharmonic waves have attracted much attention due to important applications in thin plate elasticity, which include offshore runway design \cite{WUW2004}, seismic cloaks \cite{FGE09, SWW12}, and platonic crystals \cite{MMM09}. Compared with the second order acoustic and electromagnetic wave equations, many direct and inverse scattering problems remain  unsolved for the fourth order biharmonic wave equation \cite{GGW10, S00}.

In this paper, we consider the biharmonic wave equation with a random potential
\begin{equation}\label{eq:model}
\Delta^2u-(k^2+{\rm i}\sigma k)u+\rho u=-\delta_y\quad \text{in}~\mathbb R^d,
\end{equation}
where $d=2$ or $3$, $k>0$ is the wavenumber, $\sigma\ge0$ is the damping coefficient, and $\delta_y(x):=\delta(x-y)$ denotes the point source located at $y\in\mathbb R^d$ with $\delta$ being the Dirac delta distribution. The term $\rho u$ describes physically an external linear load added to the system and represents a multiplicative noise from the point of view of stochastic partial differential equations. Denote by $\kappa=\kappa(k)$ the complex-valued wavenumber which is given by  
\begin{align*}\label{eq:kappa}
\kappa^4=k^2+{\rm i}\sigma k.
\end{align*}
Let $\kappa_{\rm r}:=\Re(\kappa)>0$ and $\kappa_{\rm i}:=\Im(\kappa)\ge0$, where $\Re(\cdot)$ and $\Im(\cdot)$ denote the real and imaginary parts of a complex number, respectively. As an outgoing wave condition for the fourth order equation, the Sommerfeld radiation condition is imposed to both the wave field $u$ and its Laplacian $\Delta u$: 
\begin{equation}\label{eq:radiation}
\lim_{r\to\infty}r^{\frac{d-1}2}\left(\partial_r u-{\rm i}\kappa u\right)=0,\quad \lim_{r\to\infty}r^{\frac{d-1}2}\left(
\partial_r\Delta u-{\rm i}\kappa\Delta u\right)=0,\quad r=|x|.
\end{equation}

The potential $\rho$ is assumed to be a Gaussian random field defined in a complete probability space $(\Omega,\mathcal F,\mathbb P)$, where $\mathbb P$ is the probability measure. More precisely, $\rho$ is required to satisfy the following assumption (cf. \cite{LPS08}).

\begin{assumption}\label{as:rho}
Let the potential $\rho$ be a real-valued centered microlocally isotropic Gaussian random field of order $m\in(d-1,d]$ in a bounded domain $D\subset\mathbb R^d$, i.e., the covariance operator $Q_\rho$ of $\rho$ is a classical pseudo-differential operator with the principal symbol $\mu(x)|\xi|^{-m}$, where $\mu$ is called the correlation strength of $\rho$ and satisfies $\mu\in C_0^\infty(D), \mu\ge0$.
\end{assumption}

Apparently, the regularity of the microlocally isotropic Gaussian random potential depends on the order $m$. It has been proved in \cite[Lemma 2.6]{LW} that the potential is relatively regular and satisfies $\rho\in C^{0,\alpha}(D)$ with $\alpha\in(0,\frac{m-d}2)$ if $m\in(d,d+2)$; the potential is rough and satisfies $\rho\in W^{\frac{m-d}2-\epsilon,p}(D)$ with $\epsilon>0$ and $p>1$ if $m\le d$. This work focuses on the rough case, i.e., $m\leq d$. Given the rough potential $\rho$, the direct scattering problem is to study the well-posedness and examine the regularity of the solution to \eqref{eq:model}--\eqref{eq:radiation}; the inverse scattering problem is to determine the correlation strength $\mu$ of the random potential $\rho$ from some statistics of the wave field $u$ satisfying \eqref{eq:model}--\eqref{eq:radiation}. Both the direct and inverse scattering problems are challenging. The unique continuation principle 
is crucial for the well-posedenss of the direct scattering problem. But it is nontrivial for the biharmonic wave equation with a rough potential. Moreover, the inverse scattering problem is nonlinear. 

The inverse random potential scattering problems were considered in \cite{CHL19,LPS08,LLW,LLW2,LLM21} for the second order wave equations, where random potentials are assumed to be microlocally isotropic Gaussian fields satisfying Assumption \ref{as:rho} with different conditions on the order $m$. For the Schr\"odinger equation, the unique continuation principle was extended in \cite{LPS08} from the integrable potential $\rho\in L^p(D)$ with $p\in(1,\infty]$ (cf. \cite{P04,KRS87,JK85}) to the rough potential $\rho\in W^{-\epsilon,p}(D)$, i.e., $m=d$. The uniqueness was also established for the two-dimensional inverse problem with $m\in[d,d+1)$. It was shown that the strength $\mu$ of the random potential $\rho$ can be uniquely determined by a single realization of the near-field data almost surely. The corresponding three-dimensional inverse problem with $m=d$ was studied in \cite{CHL19} by using the far-field pattern of the scattered field. In \cite{LLM21}, the authors considered a generalized setting for the three-dimensional Schr\"odinger equation, where both the potential and source are random. The uniqueness was obtained to determine the strength of the potential and source simultaneously based on far-field patterns. Recently, the unique continuation principle was proved in \cite{LW21} for the second order elliptic operators with rougher potentials or medium parameters of order $m\in(d-1,d]$. In \cite{LLW}, the rough model was taken to study the inverse random potential problem for the two-dimensional elastic wave equation. It was shown that the correlation strength of the random potential is uniquely determined by the near-field data under the assumption $m\in(d-\frac13,d]$. For the three-dimensional elastic wave equation, due to the lack of decay property of the fundamental solution with respect to the frequency, the far-field data was utilized in \cite{LLW2} to uniquely determine the strength of the random potential under the condition $m\in(d-\frac15,d]$.

In the deterministic setting, the unique continuation principle was investigated in \cite{B55} and \cite{P58} for the general higher order linear elliptic operators with a weak vanishing assumption and for the biharmonic operator with a nonlinear coefficient satisfying a Lipschitz-type condition, respectively. In \cite{KLU14}, the authors studied the inverse boundary value problem of determining a first order perturbation for the polyharmonic operator $(-\Delta)^n,n\ge2$ by using the Cauchy data. It was shown in \cite{KLU12} that the first order perturbation of the biharmonic operator in a bounded domain can be uniquely determined from the knowledge of the Dirichlet-to-Neumann map given on a part of the boundary. We refer to \cite{TH17,TS18,Y14,I88} and references therein for related direct and inverse scattering problems of the biharmonic operators with regular potentials. To the best of our knowledge, the unique continuation principle is not available for the biharmonic wave equation with rough potentials. 

This paper is concerned with the direct and inverse random potential scattering problems for the two- and three-dimensional biharmonic wave equation. The work contains two main contributions. First, the unique continuation principle is proved for the biharmonic wave equation with a rough potential and the well-posedness is established in the distribution sense for the direct scattering problem. Second, the uniqueness is achieved for the inverse scattering problem. In particular, we show that the correlation strength of the random potential is uniquely determined by the high frequency limit of the second moment of the scattered field averaged over the frequency band. When the medium is lossless, i.e., the damping coefficient $\sigma=0$, we demonstrate that the expectation in the data can be removed and the data of a single realization suffices for the uniqueness of the inverse problem with probability one.

As pointed out, the configurations are different for the inverse scattering problems of the second order wave equations in two and three dimensions. The two-dimensional problems make use of the point source illumination and near-field data, while the three-dimensional problems have to adopt the plane wave incidence and far-field pattern. Due to the high regularity of the fundamental solution to the biharmonic wave equation, the inverse scattering problems can be handled in a unified approach in both the two and three dimensions by employing the same configuration. This paper focuses on the former, i.e., the point source illumination and near-field data. The latter can be considered similarly and is left for a future work. Furthermore, the additional restriction on the order $m$, which was considered in \cite{LLW,LLW2}, can be withdrawn for the biharmonic wave equation. It is worth mentioning that the range of the order $m$ imposed for the inverse scattering problem is optimal in the sense that it coincides with the range of $m$ required in the unique continuation principle to ensure the well-posedness of the direct scattering problem. Our main result for the inverse scattering problem is summarized as follows.

\begin{theorem}\label{tm:main}
Let $\rho$ be a random potential satisfying Assumption \ref{as:rho} and $U\subset\mathbb R^d$ be a bounded and convex domain having a positive distance to the support $D$ of $\rho$. For any $x\in U$, the scattered field $u^s$ satisfies
\begin{align}\label{eq:main}
\lim_{K\to\infty}\frac1K\int_1^K\kappa_{\rm r}^{m+14-2d}\mathbb E|u^s(x,k)|^2d\kappa_{\rm r}=\frac{1}{8^4\pi^{4(d-2)}}\int_D\frac1{|x-z|^{2(d-1)}}\mu(z)dz=:T_d(x).
\end{align}
In addition, if the medium is lossless, i.e., $\sigma=0$, then it holds 
\begin{align}\label{eq:main2}
\lim_{K\to\infty}\frac1{2K}\int_1^{K^2}k^{\frac{m+13}2-d}|u^s(x,k)|^2dk=T_d(x)\quad\mathbb P\text{-}a.s.
\end{align}
Moreover, the strength $\mu$ of the random potential $\rho$ can be uniquely determined by $\{T_d(x)\}_{x\in U}$.
\end{theorem}

Hereafter, we use the notation ``$\mathbb P\text{-}a.s.$'' to indicate that the formula holds with probability one.
The notation $a\lesssim b$ stands for $a\le Cb$, where $C$ is a positive constant and may change from line to line in the proofs. 

The rest of the paper is organized as follows. Section \ref{sec:2} introduces the fundamental solution to the biharmonic wave equation. Section \ref{sec:3} presents the unique continuation principle for the biharmonic wave equation with rough potentials. Based on the Lippmann--Schwinger integral equation, the well-posedness for the direct scattering problem is addressed in section \ref{sec:4}. Section \ref{sec:inverse} is devoted the uniqueness of the inverse scattering problem. The paper is concluded with some general remarks in section \ref{sec:con}. 

\section{Preliminaries}
\label{sec:2}

In this section, we introduce the fundamental solution to the two- and three-dimensional biharmonic wave equation and examine some important properties of the integral operators defined by the fundamental solution. 

\subsection{The fundamental solution}

Recalling $\kappa^4=k^2+{\rm i}\sigma k$, we have from a straightforward calculation that 
\begin{align*}
\kappa_{\rm r}&=\Re(\kappa)=\Bigg[\bigg(\frac{k^4+\sigma^2k^2}{16}\bigg)^{\frac14}+\bigg(\frac{\sqrt{k^4+\sigma^2k^2}+k^2}8\bigg)^{\frac12}\Bigg]^{\frac12},\\
\kappa_{\rm i}&=\Im(\kappa)=\Bigg[\bigg(\frac{k^4+\sigma^2k^2}{16}\bigg)^{\frac14}-\bigg(\frac{\sqrt{k^4+\sigma^2k^2}+k^2}8\bigg)^{\frac12}\Bigg]^{\frac12}.
\end{align*}
It is clear to note that 
\begin{align*}
k^{\frac12}\kappa_{\rm i}=k^{\frac12}\left[\frac{\frac{\sqrt{k^4+\sigma^2k^2}-k^2}8}{\big(\frac{k^4+\sigma^2k^2}{16}\big)^{\frac14}+\big(\frac{\sqrt{k^4+\sigma^2k^2}+k^2}8\big)^{\frac12}}\right]^{\frac12}
=\left[\frac{\sqrt{k^4+\sigma^2k^2}-k^2}{8\big(\frac{k^4+\sigma^2k^2}{16k^4}\big)^{\frac14}+8\big(\frac{\sqrt{k^4+\sigma^2k^2}+k^2}{8k^2}\big)^{\frac12}}\right]^{\frac12},
\end{align*}
where
\[
\lim_{k\to\infty}\left(\sqrt{k^4+\sigma^2k^2}-k^2\right)=\lim_{k\to\infty}\frac{\sigma^2k^2}{\sqrt{k^4+\sigma^2k^2}+k^2}=\frac{\sigma^2}2.
\]
Hence we get
\begin{align}\label{eq:kappari}
\lim_{k\to\infty}\frac{\kappa_{\rm r}}{k^{\frac12}}=1,\quad \lim_{k\to\infty}k^{\frac12}\kappa_{\rm i}=\frac{\sigma}4,
\end{align}
which implies for sufficiently large $k$ that the following quantities are equivalent:
\[
|\kappa| \sim\kappa_{\rm r}\sim k^{\frac12}. 
\]

Let $\Phi(x,y,k)$ be the fundamental solution to the biharmonic wave equation, i.e., it satisfies 
\begin{align*}
\Delta^2\Phi(x,y,k)-\kappa^4\Phi(x,y,k)=-\delta(x-y).
\end{align*}
It follows from the identity $\Delta^2-\kappa^4=(\Delta+\kappa^2)(\Delta-\kappa^2)$ that $\Phi$ is a linear combination of the fundamental solutions to the Helmholtz operator  $\Delta+\kappa^2$ and the modified Helmholtz operator $\Delta-\kappa^2$ (cf. \cite{TH17,TS18}): 
\begin{align*}
\Phi(x,y,k)=-\frac{\rm i}{8\kappa^2}\left(\frac{\kappa}{2\pi|x-y|}\right)^{\frac{d-2}2}\left(H_{\frac{d-2}2}^{(1)}(\kappa|x-y|)+\frac{2\rm i}{\pi}K_{\frac{d-2}2}(\kappa|x-y|)\right),
\end{align*}
where $H_{\nu}^{(1)}$ and $K_{\nu}$ are the Hankel function of the first kind and the Macdonald function with order $\nu\in\mathbb R$, respectively. Noting 
\begin{align*}
K_{\nu}(z)=\frac{\pi}2{\rm i}^{\nu+1}H_{\nu}^{(1)}({\rm i}z),\quad -\pi<\arg z\le\frac{\pi}2
\end{align*}
and 
\[
H_{\frac12}^{(1)}(z)=\sqrt{\frac2{\pi z}}\frac{e^{{\rm i}z}}{\rm i}, 
\]
we have 
\begin{equation}\label{eq:Phi}
\Phi(x,y,k)=\left\{
\begin{aligned}
&-\frac{\rm i}{8\kappa^2}\big(H_0^{(1)}(\kappa|x-y|)-H_0^{(1)}({\rm i}\kappa|x-y|)\big),\quad& d=2,\\
&-\frac1{8\pi \kappa^2|x-y|}\big(e^{{\rm i}\kappa|x-y|}-e^{-\kappa|x-y|}\big),\quad& d=3. 
\end{aligned}
\right.
\end{equation}

The following lemma gives the regularity of $\Phi$ and its dependence on the wavenumber $k$. 

\begin{lemma}\label{lm:Phi}
Let $G\subset\mathbb R^d$ be any bounded domain with a locally Lipschitz boundary. For any fixed $y\in\mathbb R^d$, it holds $\Phi(\cdot,y,k)\in W^{\gamma,q}(G)$ for any $\gamma\in[0,1]$ and $q\in(1,\frac2{\gamma})$.
In particular, for any fixed $y\in D$ and $G$ having a positive distance from $D$, it holds for sufficiently large $k$ that
\begin{equation*}
\|\Phi(\cdot,y,k)\|_{W^{\gamma,q}(G)}\lesssim k^{\frac{d-7}4+\frac\gamma2}
\end{equation*}
for any $\gamma\in[0,1]$ and $q>1$.
\end{lemma}

\begin{proof}
Let $r^*:=\sup_{x\in G}|x-y|$ for any $y\in\mathbb R^d$ and $r_0:=\inf_{x\in G}|x-y|>0$ if $y\in D$. We discuss the two- and three-dimensional problems separately. 

First we consider the two-dimensional case, where the fundamental solution takes the form 
$
\Phi(x,y,k)=-\frac{\rm i}{8\kappa^2}(H_0^{(1)}(\kappa|x-y|)+\frac{2\rm i}{\pi}K_0(\kappa|x-y|))
$ for any fixed $y\in\mathbb R^2$.

By \cite[Lemmas 2.1 and 2.2]{CL05}, it holds for any $z\in\mathbb C$ that 
\begin{align}\label{eq:Hv}
&\big|H_\nu^{(1)}(z)\big| \le e^{-\Im(z)\left(1-\frac{\Theta^2}{|z|^2}\right)^\frac12}\big|H_\nu^{(1)}(\Theta)\big|,\\\label{eq:Kv}
&\left|K_\nu(z)\right| \le \frac{\pi}2e^{-\Re(z)\left(1-\frac{\Theta^2}{|z|^2}\right)^\frac12}\big|H_\nu^{(1)}(\Theta)\big|, 
\end{align}
where $\nu\in\mathbb R$ and $\Theta$ is any real number satisfying $0<\Theta\le|z|$. Choosing $z=\kappa|x-y|$ and $\Theta=\Re(z)=\kappa_{\rm r}|x-y|$, we get
\begin{align*}
\int_{G}\left|\Phi(x,y,k)\right|^pdx\lesssim|\kappa|^{-2p}\int_G\big|H_0^{(1)}(\kappa_{\rm r}|x-y|)\big|^pdy
\lesssim\kappa_{\rm r}^{-2p}\int_0^{r^*}\big|H_0^{(1)}(\kappa_{\rm r}r)\big|^prdr.
\end{align*} 
It follows from $H_0^{(1)}(\kappa_{\rm r}r)\sim\frac{2{\rm i}}{\pi}\ln(\kappa_{\rm r}r)$ as $r\to0$ (cf. \cite[Section 9.1.8]{AS92}) and
\[
\int_0^{r^*}|\ln(\kappa_{\rm r}r)|^prdr=\kappa_{\rm r}^{-2}\int_0^{\kappa_{\rm r}r^*}|\ln(r)|^prdr\lesssim\kappa_{\rm r}^{2p\epsilon}\quad\forall\, p>1, \epsilon>0
\]
that the following estimate holds:
\begin{align*}
\|\Phi(\cdot,y,k)\|_{L^{p}(G)}\lesssim\kappa_{\rm r}^{-2+2\epsilon}<\infty\quad\forall\, p>1, \epsilon>0.
\end{align*}

Noting 
\begin{align*}
&\partial_{x_i}H_0^{(1)}(\kappa|x-y|)=\kappa H_0^{(1)'}(\kappa|x-y|)\frac{x_i-y_i}{|x-y|}=-\kappa H_1^{(1)}(\kappa|x-y|)\frac{x_i-y_i}{|x-y|},\\
&\partial_{x_i}K_0(\kappa|x-y|)=\frac{{\rm i}\pi}2\partial_{x_i}H_0^{(1)}({\rm i}\kappa|x-y|)=-{\rm i}\kappa K_1(\kappa|x-y|)\frac{x_i-y_i}{|x-y|}
\end{align*}
for $i=1, 2$ and using  $H_1^{(1)}(\kappa_{\rm r}r)\sim\frac{2{\rm i}}\pi\frac1{\kappa_{\rm r}r}$ as $r\to0$ (cf. \cite[Section 9.1.9]{AS92}), we obtain 
\begin{align*}
\int_G\left|\partial_{x_i}\Phi(x,y,k)\right|^{p'}dx &\lesssim
|\kappa|^{-p'}\int_G\left|H_1^{(1)}(\kappa_{\rm r}|x-y|)\right|^{p'}dx
\lesssim\kappa_{\rm r}^{-p'}\int_0^{r^*}\left|H_1^{(1)}(\kappa_{\rm r}r)\right|^{p'}rdr\\
&\lesssim\kappa_{\rm r}^{-p'}\int_0^{r^*}\frac1{(\kappa_{\rm r}r)^{p'}}rdr\lesssim\kappa_{\rm r}^{-2p'}\quad\forall\,p'\in(1,2),
\end{align*}
which shows 
\begin{align*}
\|\Phi(\cdot,y,k)\|_{W^{1,p'}(G)}\lesssim \kappa_{\rm r}^{-2}<\infty\quad\forall\, p'\in(1,2)
\end{align*}
and $\Phi(\cdot,y,k)\in W^{1,p'}(G)$.

The interpolation $[L^p(G),W^{1,p'}(G)]_{\gamma}=W^{\gamma,q}(G)$ with $\gamma\in[0,1]$ and $q$ satisfying $\frac1q=\frac{1-\gamma}p+\frac{\gamma}{p'}$ (cf. \cite[Theorem 6.4.5]{BL76}) yields $\Phi(\cdot,y,k)\in W^{\gamma,q}(G)$ for any $\gamma\in[0,1]$ and $q\in(1,\frac2{\gamma})$.

In particular, if $y\in D$ and $k$ is sufficiently large, then $r_0:=\inf_{x\in G}|x-y|>0$ and the Hankel function has the following asymptotic expansion (cf. \cite[Section 9.2.3]{AS92}): 
\[
H_{\nu}^{(1)}(\kappa_{\rm r}|x-y|)\sim\left(\frac2{\pi\kappa_{\rm r}|x-y|}\right)^{\frac12}e^{{\rm i}(\kappa_{\rm r}|x-y|-\frac12\nu\pi-\frac14\pi)}
\] 
for $\nu\in\mathbb R$. Following from the interpolation between $L^q(G)$ and $W^{1,q}(G)$ provided that $G$ is bounded with a locally Lipschitz boundary (cf. \cite[Section 9.69]{AF03}), we have 
\begin{align*}
\int_{G}\left|\Phi(x,y,k)\right|^qdx\lesssim|\kappa|^{-2q}\int_G\big|H_0^{(1)}(\kappa_{\rm r}|x-y|)\big|^qdy
\lesssim\kappa_{\rm r}^{-2q}\int_{r_0}^{r^*}\frac1{(\kappa_rr)^{\frac q2}}rdr\lesssim\kappa_{\rm r}^{-\frac52q},
\end{align*} 
\begin{align*}
\int_G\left|\partial_{x_i}\Phi(x,y,k)\right|^{q}dx\lesssim
|\kappa|^{-q}\int_G\big|H_1^{(1)}(\kappa_{\rm r}|x-y|)\big|^{q}dx
\lesssim\kappa_{\rm r}^{-q}\int_{r_0}^{r^*}\frac1{(\kappa_{\rm r}r)^{\frac q2}}rdr\lesssim\kappa_{\rm r}^{-\frac32q},
\end{align*}
which leads to
\begin{align}\label{eq:Phiesti}
\|\Phi(\cdot,y,k)\|_{W^{\gamma,q}(G)}\lesssim \kappa_{\rm r}^{-\frac52+\gamma}\lesssim k^{-\frac54+\frac\gamma2}
\end{align}
for any $\gamma\in[0,1]$ and $q>1$. 

Next we examine the three-dimensional problem, where $\Phi(x,y,k)=-\frac1{8\pi \kappa^2|x-y|}\left(e^{{\rm i}\kappa|x-y|}-e^{-\kappa|x-y|}\right)$. The estimates are similar to the two-dimensional case.

For any $y\in\mathbb R^3$, it holds
\begin{align*}
\|\Phi(\cdot,y,k)\|_{L^q(G)}\lesssim|\kappa|^{-2}\bigg(\int_0^{r^*}\frac{|e^{{\rm i}\kappa r}-e^{-\kappa r}|^q}{r^q}r^2dr\bigg)^{\frac1q}\lesssim|\kappa|^{-2}\bigg(\int_0^{r^*}\frac{|\kappa r|^q}{r^q}r^2dr\bigg)^{\frac1q}<\infty\quad\forall~q>1.
\end{align*}
The derivatives of $\Phi$ satisfy
\begin{align*}
\int_G|\partial_{x_i}\Phi(x,y,k)|^{q}dx &=\int_G\left|\frac{x_i-y_i}{8\pi\kappa^2|x-y|^3}\left[e^{{\rm i}\kappa|x-y|}({\rm i}\kappa|x-y|-1)+e^{-\kappa|x-y|}(\kappa|x-y|+1)\right]\right|^{q}dx\\
&\lesssim|\kappa|^{-2q}\int_0^{r^*}\frac{|e^{{\rm i}\kappa r}({\rm i}\kappa r-1)+e^{-\kappa r}(\kappa r+1)|^{q}}{r^{2{q}}}r^2dr
<\infty\quad\forall~q>1,
\end{align*}
which implies $\Phi(\cdot,y,k)\in W^{\gamma,q}(G)$ for any $\gamma\in[0,1]$ and $q>1$.  

In particular, for $y\in D$, a straightforward calculation gives 
\[
\|\Phi(\cdot,y,k)\|_{L^q(G)}\lesssim|\kappa|^{-2}\bigg(\int_{r_0}^{r^*}\frac{|e^{{\rm i}\kappa r}-e^{-\kappa r}|^q}{r^q}r^2dr\bigg)^{\frac1q}\lesssim|\kappa|^{-2}\bigg(\int_{r_0}^{r^*}r^{2-q}dr\bigg)^{\frac1q}\lesssim |\kappa|^{-2},
\]
\[
\int_G|\partial_{x_i}\Phi(x,y,k)|^{q}dx=|\kappa|^{-2q}\int_{r_0}^{r^*}\frac{|\kappa r|^{q}+1}{r^{2q}}r^2dr\lesssim |\kappa|^{-q},
\]
Hence, for sufficiently large $k$, it holds
\[
\|\Phi(\cdot,y,k)\|_{W^{\gamma,q}(G)}\lesssim |\kappa|^{-2+\gamma}\lesssim k^{-1+\frac\gamma2}
\] 
for any $\gamma\in[0,1]$ and $q>1$.
\end{proof}

\subsection{Integral operators}

Define the integral operators 
\begin{align*}
\mathcal H_k(\phi)(\cdot):&=\int_{\mathbb R^d}\Phi(\cdot,z,k)\phi(z)dz,\\
\mathcal K_k(\phi)(\cdot):&=\mathcal H_k(\rho\phi)(\cdot)=\int_{\mathbb R^d}\Phi(\cdot,z,k)\rho(z)\phi(z)dz,
\end{align*}
where $\Phi$ is the fundamental solution given in \eqref{eq:Phi} and $\rho$ is the random potential satisfying Assumption 1. 

\begin{lemma}\label{lm:operatorH}
Let $B$ and $G$ be two bounded domains in $\mathbb R^d$, and $G$ has a locally Lipschitz boundary.
\begin{itemize}
\item[(i)] The operator $\mathcal H_k: H^{-s_1}(B)\to H^{s_2}(G)$ is bounded and satisfies
\[
\|\mathcal H_k\|_{\mathcal{L}(H^{-s_1}(B),H^{s_2}(G))}\lesssim k^{\frac{s-3}2}
\]
for $s:=s_1+s_2\in(0,3)$ with $s_1,s_2\ge0$. 
\item[(ii)] The operator $\mathcal H_k: H^{-s}(B)\to L^{\infty}(G)$ is bounded and satisfies
\[
\|\mathcal H_k\|_{\mathcal L(H^{-s}(B),L^\infty(G))}\lesssim k^{\frac{2s+d-6+\epsilon}4}
\]
for any $s\in(0,3)$ and $\epsilon>0$.
\item[(iii)] The operator $\mathcal H_k:W^{-\gamma,p}(B)\to W^{\gamma,q}(G)$ is compact for any $q>1$, $0<\gamma<\min\{\frac32,\frac32+(\frac1q-\frac12)d\}$, and $p>1$ with $p$ and $q$ satisfying $\frac1p+\frac1q=1$.
\end{itemize}
\end{lemma}

\begin{proof}
(i) Since the case $\sigma=0$ is discussed in \cite[Lemma 3.1]{LW2}, we only show the proof for the case $\sigma>0$ where $\kappa_{\rm i}>0$. For any two smooth test functions $\phi\in C_0^{\infty}(B)$ and $\psi\in C_0^{\infty}(G)$, we consider
\begin{align}\label{eq:esti1}
|\langle\mathcal H_k(\phi),\psi\rangle|&=\bigg|\int_{\mathbb R^d}\frac1{|\xi|^4-\kappa^4}\hat\phi(\xi)\hat\psi(\xi)d\xi\bigg|\notag\\
&=\bigg|\int_{\mathbb R^d}\frac{(1+|\xi|^2)^{\frac s2}}{(|\xi|^2+\kappa^2)(|\xi|+\kappa)(|\xi|-\kappa)}\widehat{\mathcal J^{-s_1}\phi}(\xi)\widehat{\mathcal J^{-s_2}\psi}(\xi)d\xi\bigg|,
\end{align}
where $\hat\phi$ and $\hat\psi$ are the Fourier transform of $\phi$ and $\psi$, respectively, and $\mathcal J^{-s}$ stands for the Bessel potential of order $-s$ and is defined by (cf. \cite{LW21})
\[
\mathcal J^{-s}f:=\mathcal F^{-1}\big((1+|\cdot|^2)^{-\frac s2}\hat f\big)
\]
with $\mathcal F^{-1}$ denoting the inverse Fourier transform. 

The integral domain $\mathbb R^d$ of \eqref{eq:esti1} can be split into two parts
\begin{align*}
\Omega_1:=\left\{\xi\in\mathbb R^d:||\xi|-\kappa_{\rm r}|>\frac{\kappa_{\rm r}}2\right\},\quad 
\Omega_2:=\left\{\xi\in\mathbb R^d:||\xi|-\kappa_{\rm r}|<\frac{\kappa_{\rm r}}2\right\}.
\end{align*}
Following the same procedure as \cite[Lemma 3.1]{LW2} and using \eqref{eq:kappari}, we get
\[
|\langle\mathcal H_k(\phi),\psi\rangle|\lesssim\kappa_{\rm r}^{s-3}\|\phi\|_{H^{-s_1}(B)}\|\psi\|_{H^{-s_2}(G)}\lesssim k^{\frac{s-3}2}\|\phi\|_{H^{-s_1}(B)}\|\psi\|_{H^{-s_2}(G)},
\]
which completes the proof by extending the above result to $\phi\in H^{-s_1}(B)$ and $\psi\in H^{-s_2}(G)$.

(ii) For any $\phi\in C_0^\infty(B)$, we still denote by $\phi$ its zero extension outside of $B$. It follows from the Plancherel theorem that 
\begin{align*}
\mathcal H_k(\phi)(x)&=\int_{\mathbb R^d}\Phi(x,z,k)\phi(z)dz\\
&=\int_{\mathbb R^d}(1+|\xi|^2)^{\frac s2}\widehat{\Phi}(x,\xi,k)\widehat{\mathcal J^{-s}\phi}(\xi)d\xi,\\
&=-\int_{\mathbb R^d}\frac{(1+|\xi|^2)^{\frac{2s+d+\epsilon}4}}{|\xi|^4-\kappa^4}\widehat{\mathcal J^{-s}\phi}(\xi)\big(e^{-{\rm i}x\cdot\xi}(1+|\xi|^2)^{-\frac{d+\epsilon}4}\big)d\xi,
\end{align*}
where 
\[
\widehat{\Phi}(x,\xi,k):=\mathcal F[\Phi(x,\cdot,k)](\xi)=\frac{-e^{-{\rm i}x\cdot\xi}}{|\xi|^4-\kappa^4}
\] 
is the Fourier transform of $\Phi(x,y,k)$ with respect to $y$. Comparing the above integral with \eqref{eq:esti1} and replacing $\widehat{\mathcal J^{-s_2}\psi}(\xi)$ by $g(\xi):=e^{-{\rm i}x\cdot\xi}(1+|\xi|^2)^{-\frac{d+\epsilon}4}$, we obtain 
\begin{align*}
|\mathcal H_k(\phi)(x)|
\lesssim k^{\frac{\frac{2s+d+\epsilon}2-3}2}\|\phi\|_{H^{-s}(B)}\lesssim k^{\frac{2s+d-6+\epsilon}4}\|\phi\|_{H^{-s}(B)},
\end{align*}
which can also be extended to $\phi\in H^{-s}(B)$. We mention that $g\in H^1(\mathbb R^d)$ is utilized in the above estimate, which is required in the estimate of \eqref{eq:esti1} (see e.g., \cite{LW21,LLW}).

(iii) The compactness of $\mathcal H_k$ can be obtained from the boundedness shown in (i) and the Sobolev embedding theorem. In fact,  according to the Kondrachov embedding theorem, the embeddings
\begin{align*}
&W^{-\gamma,p}(B)\hookrightarrow H^{-s_1}(B),\\
&H^{s_2}(G)\hookrightarrow W^{\gamma,q}(G)
\end{align*}
are continuous under conditions
\begin{align*}
&\gamma<s_1,\quad\frac12>\frac1p-\frac{s_1-\gamma}{d},\\
&\gamma<s_2,\quad\frac1q>\frac12-\frac{s_2-\gamma}d,
\end{align*}
and $s_1+s_2\in(0,3)$. It is easy to check that the above conditions are satisfied if $\frac1p+\frac1q=1$ and 
\[
0<\gamma<\min\left\{\frac{s_1+s_2}2,\frac{s_1+s_2}2+d\left(\frac1q-\frac12\right)\right\},
\]
which completes the proof of (iii) due to $s_1+s_2<3$.
\end{proof}

The estimates for the operator $\mathcal K_k$ can be obtained from the estimates of $\mathcal H_k$ given in Lemma \ref{lm:operatorH} and the relation $\mathcal K_k(\phi)=\mathcal H_k(\rho\phi)$. 

\begin{lemma}\label{lm:operatorK}
Let $G\subset\mathbb R^d$ be a bounded domain with a locally Lipschitz boundary and the random potential $\rho$ satisfy Assumption \ref{as:rho}.
\begin{itemize}
\item[(i)]  The operator $\mathcal K_k:W^{\gamma,q}(G)\to W^{\gamma,q}(G)$ is compact for any $q\in(2,A)$ and $\gamma\in(\frac{d-m}2,\frac32+(\frac1q-\frac12)d)$ with
\[
A:=\left\{\begin{aligned}
&\infty,\quad &&d=2\quad\text{or}\quad m=d=3,\\
&\frac{6}{3-m},\quad&&m<d=3,
\end{aligned}\right.
\]
and satisfies
\[
\|\mathcal K_k\|_{\mathcal L(W^{\gamma,q}(G))}\lesssim k^{\gamma+(\frac12-\frac1q)d-\frac32}\quad\mathbb P\text{-}a.s.
\]
\item[(ii)] The following estimates hold: 
\[
\|\mathcal K_k\|_{\mathcal L(H^{s}(G))}\lesssim k^{s-\frac32}\quad\mathbb P\text{-}a.s.
\]
for any $s\in(\frac{d-m}2,\frac32)$ and
\[
\|\mathcal K_k\|_{\mathcal L(H^s(G),L^{\infty}(G))}\lesssim k^{\frac{2s+d-6+\epsilon}4}\quad\mathbb P\text{-}a.s.
\]
for any $s\in(\frac{d-m}2,3)$ and $\epsilon>0$.
\end{itemize}
\end{lemma}

\begin{proof}
(i) Under Assumption \ref{as:rho}, it holds that $\rho\in W^{\frac{m-d}2-\epsilon,p'}(D)$ for any $\epsilon>0$ and $p'>1$ based on \cite[Lemma 2.2]{LW2}. 
Then for any $m\in(d-1,d]$, $q\in(2,A)\neq\emptyset$ and $\gamma\in(\frac{d-m}2,\frac32+(\frac1q-\frac12)d)\neq\emptyset$, there exists some $p'>1$ such that the embedding
\[
W^{\frac{m-d}2-\epsilon,p'}(D)\hookrightarrow W^{-\gamma,\tilde p}(D)
\]
is continuous with $\tilde p=\frac{q}{q-2}>1$. Moreover, for any $\phi\in W^{\gamma,q}(G)$, we have from \cite[Lemma 2]{LPS08} that $\rho \phi\in W^{-\gamma,p}(D)$ with $\frac1p+\frac1q=1$ and
\begin{align}\label{eq:rhophi}
\|\rho\phi\|_{W^{-\gamma,p}(D)}\lesssim\|\rho\|_{W^{-\gamma,\tilde p}(D)}\|\phi\|_{W^{\gamma,q}(G)}.
\end{align}
Hence
\[
\|\mathcal K_k(\phi)\|_{W^{\gamma,q}(G)}\lesssim\|\mathcal H_k\|_{\mathcal L(W^{-\gamma,p}(D),W^{\gamma,q}(G))}\|\rho\phi\|_{W^{-\gamma,p}(D)}\quad\mathbb P\text{-}a.s.,
\] 
which implies the compactness of $\mathcal K_k$ due to the compactness of $\mathcal H_k$ proved in Lemma \ref{lm:operatorH}.

To estimate the operator norm, we choose $s=\gamma+(\frac12-\frac1q)d$ such that the embeddings
\begin{equation}\label{eq:embedding}
\begin{aligned}
H^s(G)\hookrightarrow W^{\gamma,q}(G),\\
W^{-\gamma,p}(D)\hookrightarrow H^{-s}(D)
\end{aligned}
\end{equation}
hold with $p$ satisfying $\frac1p+\frac1q=1$. The result is obtained by noting 
\begin{align*}
\|\mathcal K_k(\phi)\|_{W^{\gamma,q}(G)} &\lesssim\|\mathcal K_k(\phi)\|_{H^s(G)}\lesssim\|\mathcal H_k\|_{\mathcal L(H^{-s}(D),H^s(G))}\|\rho\phi\|_{H^{-s}(D)}\\
&\lesssim\|\mathcal H_k\|_{\mathcal L(H^{-s}(D),H^s(G))}\|\rho\phi\|_{W^{-\gamma,p}(D)}\\
&\lesssim k^{\gamma+(\frac12-\frac1q)d-\frac32}\|\phi\|_{W^{\gamma,q}(G)}.
\end{align*}

(ii) For any $\phi\in H^s(G)$ with $s>\frac{d-m}2$, there exist $\gamma\in(\frac{d-m}2,s)$ and $q\in(2,A)$ satisfying $\frac1q>\frac12-\frac{s-\gamma}2$ such that the embeddings \eqref{eq:embedding} hold. It follows from Lemma \ref{lm:operatorH} and \eqref{eq:rhophi} that we have 
\begin{align}\label{eq:K}
\|\mathcal K_k(\phi)\|_{H^s(G)}&\lesssim\|\mathcal H_k\|_{\mathcal L(H^{-s}(D),H^s(G))}\|\rho\phi\|_{H^{-s}(D)}\lesssim\|\mathcal H_k\|_{\mathcal L(H^{-s}(D),H^s(G))}\|\rho\phi\|_{W^{-\gamma,p}(D)}\notag\\
&\lesssim k^{\frac{2s-3}2}\|\rho\|_{W^{-\gamma,\tilde p}(D)}\|\phi\|_{W^{\gamma,q}(G)}
\lesssim k^{s-\frac32}\|\phi\|_{H^s(G)}\quad\mathbb P\text{-}a.s.
\end{align}
with $s\in(\frac{d-m}2,\frac32)$, and
\begin{align*}
\|\mathcal K_k(\phi)\|_{L^{\infty}(G)}\lesssim&\|\mathcal H_k\|_{\mathcal L(H^{-s}(D),L^{\infty}(G))}\|\rho\phi\|_{H^{-s}(D)}\lesssim k^{\frac{2s+d-6+\epsilon}4}\|\phi\|_{H^s(G)}\quad\mathbb P\text{-}a.s.
\end{align*}
with $s\in(\frac{d-m}2,3)$ and $\epsilon>0$.
\end{proof}

\section{The unique continuation}\label{sec:3}

This section is to investigate the unique continuation principle, which is essential for the uniqueness of the solution to the biharmonic wave scattering problem with a random potential. We refer to \cite{LPS08,LW21} for the unique continuation of the solutions to the stochastic acoustic and elastic wave equations.

\begin{theorem}\label{tm:unique}
Let $\rho$ be a distribution satisfying Assumption \ref{as:rho}, $q\in(2,\frac{2d}{3d-2m-2})$ and $\gamma\in(\frac{d-m}2,\frac12+(\frac1q-\frac12)\frac d2)$. If $u\in W^{\gamma,q}(\mathbb R^d)$ is compactly supported in $\mathbb R^d$ and is a distributional solution to the homogeneous biharmonic wave equation 
\[
\Delta^2u-\kappa^4u+\rho u=0,
\]
then $u\equiv0$ in $\mathbb R^d$.
\end{theorem}

\begin{proof}
We consider an auxiliary function $v(x):=e^{-{\rm i}\eta\cdot x}u(x)$, where the complex vector $\eta$ is defined by
\[
\eta:=(\omega t,0,\cdots,0,\eta_d)^\top,\quad t\gg1,
\]
where 
\[
\omega:=\bigg(\frac{\sqrt{k^4+\sigma^2k^2}+k^2}2\bigg)^{\frac14}
\] 
and $\eta_d=\eta_d^{\rm r}+{\rm i}\eta_d^{\rm i}$ with the real and imaginary parts being given by 
\begin{align*}
\eta_d^{\rm r}&=\bigg(\frac{\sqrt{\omega^4(t^2-1)^2+\omega^4-k^2}-\omega^2(t^2-1)}2\bigg)^{\frac12},\\
\eta_d^{\rm i}&=\bigg(\frac{\sqrt{\omega^4(t^2-1)^2+\omega^4-k^2}+\omega^2(t^2-1)}2\bigg)^{\frac12},
\end{align*}
respectively. It is clear to note $\eta\cdot\eta=\kappa^2=\omega^2+{\rm i}(\omega^4-k^2)^{\frac12}$. Moreover, a simple calculation shows that 
\begin{align}\label{eq:eta}
\lim_{t\to\infty}\eta_d^{\rm r}=0,\quad\lim_{t\to\infty}\frac{\eta_d^{\rm i}}{t}=\omega.
\end{align}
Then $v$ is also compactly supported in $\mathbb R^d$ and satisfies
\[
\Delta^2v+4{\rm i}\eta\cdot\nabla\Delta v-4\eta^\top\nabla^2v\eta-2(\eta\cdot\eta)\Delta v-4{\rm i}(\eta\cdot\eta)(\eta\cdot\nabla v)=-\rho v.
\]
Taking the Fourier transform of the above equation yields  
\begin{align}\label{eq:v}
v=-\mathcal G_\eta(\rho v),
\end{align}
where $\mathcal G_\eta$ is defined by
\begin{align*}
\mathcal G_\eta(f)(x):=\mathcal{F}^{-1}\bigg[\frac{\hat f(\xi)}{|\xi|^4+4|\xi|^2(\eta\cdot\xi)+4(\eta\cdot\xi)^2+2(\eta\cdot\eta)|\xi|^2+4(\eta\cdot\eta)(\eta\cdot\xi)}\bigg](x),\quad \xi\in\mathbb R^d. 
\end{align*}

It suffices to show $v\equiv0$ in order to show show $u\equiv0$. The proof consists of two steps. The first step is to estimate the operator $\mathcal G_\eta$ in Hilbert spaces. 

Let $G\subset\mathbb R^d$ be a bounded domain with a locally Lipschitz boundary containing the compact supports of both $\rho$ and $u$. For any $f,g\in C_0^\infty(G)$, we denote their zero extensions outside of $G$ still by $f,g$ for simplicity. Using the Plancherel theorem, we have from a straightforward calculation that 
\begin{align}\label{eq:Geta}
\langle\mathcal G_\eta f,g\rangle=\langle\widehat{\mathcal G_\eta f},\hat g\rangle
&=\int_{\mathbb R^d}\frac{\hat f(\xi)\overline{\hat g(\xi)}}{|\xi|^4+4|\xi|^2(\eta\cdot\xi)+4(\eta\cdot\xi)^2+2\kappa^2|\xi|^2+4\kappa^2(\eta\cdot\xi)}d\xi\notag\\
&=\int_{\mathbb R^d}\frac{\hat f(\xi)\overline{\hat g(\xi)}}{(|\xi|^2+2\eta\cdot\xi+2\kappa^2)(|\xi|^2+2\eta\cdot\xi)}d\xi\notag\\
&=\frac1{2\kappa^2}\bigg[\int_{\mathbb R^d}\frac{\hat f(\xi)\overline{\hat g(\xi)}}{|\xi|^2+2\eta\cdot\xi}d\xi-\int_{\mathbb R^d}\frac{\hat f(\xi)\overline{\hat g(\xi)}}{|\xi|^2+2\eta\cdot\xi+2\kappa^2}d\xi\bigg]\notag\\
&=:\frac1{2\kappa^2}\big[\mathcal{A}-\mathcal{B}\big].
\end{align}
Denote $\xi^{-}:=(\xi_1,\cdots,\xi_{d-1})^\top\in\mathbb R^{d-1}$ and $\xi^{--}:=(\xi_2,\cdots,\xi_{d-1})^\top\in\mathbb R^{d-2}$ with $\xi^{--}=0$ if $d=2$. Then $\mathcal{A}$ can be rewritten as
\begin{align*}
\mathcal{A}
&=\int_{\mathbb R^d}\frac{\hat f(\xi)\overline{\hat g(\xi)}}{|\xi|^2+2\omega t\xi_1+2\eta_d\xi_d}d\xi\\
&=\int_{\mathbb R^d}\frac{\hat f(\xi)\overline{\hat g(\xi)}}{(\xi_1+\omega t)^2+|\xi^{--}|^2-\omega^2t^2+(\xi_d+\eta_d^{\rm r})^2-(\eta_d^{\rm r})^2+2{\rm i}\eta_d^{\rm i}\xi_d}d\xi\\
&=\int_{\mathbb R^d}\frac{\hat f(\xi)\overline{\hat g(\xi)}}{|\xi|^2-\omega^2t^2-(\eta_d^{\rm r})^2+2{\rm i}\eta_d^{\rm i}(\xi_d-\eta_d^{\rm r})}d\xi,
\end{align*}
where we used the transformation of variables $(\xi_1+\omega t,\xi_2,\cdots,\xi_d+\eta_d^{\rm r})^\top\mapsto(\xi_1,\cdots,\xi_d)^\top$ and $\hat f(\xi_1,\cdots,\xi_j-a,\cdots,\xi_d)=e^{-{\rm i}a\xi_j}\hat f(\xi)$.
Using $\kappa^2=\eta\cdot\eta=\omega^2t^2+\eta_d^2$, the transformation $(\xi_1+\omega t,\xi_2,\cdots,\xi_d+\eta_d^{\rm r})^\top\mapsto(\xi_1,\cdots,\xi_d)^\top$, and $\omega^2t^2+(\eta_d^{\rm r})^2-2(\eta_d^{\rm i})^2<0$ as $t\gg1$, we have 
\begin{align*}
\mathcal{B}&=\int_{\mathbb R^d}\frac{\hat f(\xi)\overline{\hat g(\xi)}}{(\xi_1+\omega t)^2+|\xi^{--}|^2+(\xi_d+\eta_d^{\rm r})^2+\omega^2t^2+(\eta_d^{\rm r})^2-2(\eta_d^{\rm i})^2+2{\rm i}\eta_d^{\rm i}(\xi_d+2\eta_d^{\rm r})}d\xi\\
&=\int_{\mathbb R^d}\frac{\hat f(\xi)\overline{\widehat g(\xi)}}{|\xi|^2+\omega^2t^2+(\eta_d^{\rm r})^2-2(\eta_d^{\rm i})^2+2{\rm i}\eta_d^{\rm i}(\xi_d+\eta_d^{\rm r})}d\xi. 
\end{align*}

The estimates for $\mathcal{A}$ and $\mathcal{B}$ are similar, and the integral domain $\mathbb R^d$ needs to be decomposed into several subdomains according to the singularity of the integrands. In the following, we show the detail of the estimate for $\mathcal A$. The analysis of $\mathcal B$ can be carried out analogously and is omitted for brevity. 

It is easy to see that the function 
\[
\frac{1}{|\xi|^2-\omega^2t^2-(\eta_d^{\rm r})^2+2{\rm i}\eta_d^{\rm i}(\xi_d-\eta_d^{\rm r})}=\frac{1}{|\xi^-|^2-\omega^2t^2+(\xi_d-\eta_d^{\rm r})(\xi_d+\eta_d^{\rm r})+2{\rm i}\eta_d^{\rm i}(\xi_d-\eta_d^{\rm r})}
\]
is singular on the manifold
$\{\xi\in\mathbb R^d: |\xi^-|=\omega t,~\xi_d=\eta_d^{\rm r}\}.$ Define two domains
\begin{align*}
\Omega_1:&=\left\{\xi:||\xi^-|-\omega t|>\frac{\omega t}2\right\}=\left\{
\xi:|\xi^-|>\frac{3\omega t}2\right\}\cup\left\{
\xi:|\xi^-|<\frac{\omega t}2\right\},\\
\Omega_2:&=\left\{\xi:||\xi^-|-\omega t|<\frac{\omega t}2\right\}=\left\{
\xi:\frac{\omega t}2<|\xi^-|<\frac{3\omega t}2\right\}. 
\end{align*}
Based on $\Omega_1$ and $\Omega_2$, $\mathcal{A}$ can be split into the following two terms:
\begin{align*}
\mathcal{A}&=\int_{\Omega_1}\frac{(1+|\xi|^2)^s}{|\xi|^2-\omega^2t^2-(\eta_d^{\rm r})^2+2{\rm i}\eta_d^{\rm i}(\xi_d-\eta_d^{\rm r})}\widehat{\mathcal J^{-s}f}(\xi)\overline{\widehat{\mathcal J^{-s}g}(\xi)}d\xi\\
&\quad +\int_{\Omega_2}\frac{(1+|\xi|^2)^s}{|\xi|^2-\omega^2t^2-(\eta_d^{\rm r})^2+2{\rm i}\eta_d^{\rm i}(\xi_d-\eta_d^{\rm r})}\widehat{\mathcal J^{-s}f}(\xi)\overline{\widehat{\mathcal J^{-s}g}(\xi)}d\xi\\
&=:{\rm I}+{\rm II},
\end{align*}
where $s\in(0,\frac12)$. Next is to estimate ${\rm I}$ and ${\rm II}$, respectively. 

Term ${\rm I}$ satisfies
\begin{align*}
|{\rm I}|&\le\int_{\Omega_1}\frac{(1+|\xi|^2)^s}{\left[(|\xi|^2-\omega^2t^2-(\eta_d^{\rm r})^2)^2+4(\eta_d^{\rm i})^2(\xi_d-\eta_d^{\rm r})^2\right]^{\frac12}}|\widehat{\mathcal J^{-s}f}||\widehat{\mathcal J^{-s}g}|d\xi\\
&=\int_{\{\xi:|\xi^-|>\frac{3\omega t}2\}}\frac{(1+|\xi|^2)^s}{\left[(|\xi|^2-\omega^2t^2-(\eta_d^{\rm r})^2)^2+4(\eta_d^{\rm i})^2(\xi_d-\eta_d^{\rm r})^2\right]^{\frac12}}|\widehat{\mathcal J^{-s}f}||\widehat{\mathcal J^{-s}g}|d\xi\\
&\quad+\int_{\{\xi:|\xi^-|<\frac{\omega t}2,|\xi_d-\eta_d^{\rm r}|>\frac{\omega t}2\}}\frac{(1+|\xi|^2)^s}{\left[(|\xi|^2-\omega^2t^2-(\eta_d^{\rm r})^2)^2+4(\eta_d^{\rm i})^2(\xi_d-\eta_d^{\rm r})^2\right]^{\frac12}}|\widehat{\mathcal J^{-s}f}||\widehat{\mathcal J^{-s}g}|d\xi\\
&\quad+\int_{\{\xi:|\xi^-|<\frac{\omega t}2,|\xi_d-\eta_d^{\rm r}|<\frac{\omega t}2\}}\frac{(1+|\xi|^2)^s}{\left[(|\xi|^2-\omega^2t^2-(\eta_d^{\rm r})^2)^2+4(\eta_d^{\rm i})^2(\xi_d-\eta_d^{\rm r})^2\right]^{\frac12}}|\widehat{\mathcal J^{-s}f}||\widehat{\mathcal J^{-s}g}|d\xi\\
&=:{\rm I}_1+{\rm I}_2+{\rm I}_3.
\end{align*} 
By \eqref{eq:eta}, we may choose a sufficiently large $t^*$ such that $\eta_d^{\rm r}<\frac{\omega t}4$ for all $t>t^*$, which leads to
\[
\frac{3\omega t}2-\sqrt{\omega^2t^2+(\eta_d^{\rm r})^2}>\frac{\omega t}4,\quad t>t^*.
\]
We then get 
\begin{align*}
{\rm I}_1 &\le\int_{\{\xi:|\xi|>\frac{3\omega t}2\}}\frac{(1+|\xi|^2)^s}{(|\xi|-\sqrt{\omega^2t^2+(\eta_d^{\rm r})^2})(|\xi|+\sqrt{\omega^2t^2+(\eta_d^{\rm r})^2})}|\widehat{\mathcal J^{-s}f}||\widehat{\mathcal J^{-s}g}|d\xi\\
&\lesssim\frac1{\omega t}\int_{\{\xi:|\xi|>\frac{3\omega t}2\}}\frac1{|\xi|^{1-2s}}|\widehat{\mathcal J^{-s}f}||\widehat{\mathcal J^{-s}g}|d\xi\\
&\lesssim\frac1{(\omega t)^{2-2s}}\|f\|_{H^{-s}(G)}\|g\|_{H^{-s}(G)}.
\end{align*}
Note also that $\eta_d^{\rm i}$ is equivalent to $\omega t$ as $t\to\infty$, which yields
\begin{align*}
{\rm I}_2 &\le\int_{\{\xi:|\xi^-|<\frac{\omega t}2,|\xi_d-\eta_d^{\rm r}|>\frac{\omega t}2\}}\frac{(1+\frac{\omega^2t^2}4+\xi_d^2)^s}{2\eta_d^{\rm i}|\xi_d-\eta_d^{\rm r}|}|\widehat{\mathcal J^{-s}f}||\widehat{\mathcal J^{-s}g}|d\xi\\
&\lesssim\int_{\{\xi:|\xi^-|<\frac{\omega t}2,|\xi_d-\eta_d^{\rm r}|>\frac{\omega t}2\}}\frac{(\omega t)^{2s}+|\xi_d-\eta_d^{\rm r}|^{2s}+(\eta_d^{\rm r})^{2s}}{2\eta_d^{\rm i}|\xi_d-\eta_d^{\rm r}|}|\widehat{\mathcal J^{-s}f}||\widehat{\mathcal J^{-s}g}|d\xi\\
&\lesssim\int_{\{\xi:|\xi^-|<\frac{\omega t}2,|\xi_d-\eta_d^{\rm r}|>\frac{\omega t}2\}}\left(\frac1{(\omega t)^{2-2s}}+\frac1{\omega t|\xi_d-\eta_d^{\rm r}|^{1-2s}}\right)|\widehat{\mathcal J^{-s}f}||\widehat{\mathcal J^{-s}g}|d\xi\\
&\lesssim\frac1{(\omega t)^{2-2s}}\|f\|_{H^{-s}(G)}\|g\|_{H^{-s}(G)}.
\end{align*}
Moreover, for any $\xi\in\{\xi:|\xi^-|<\frac{\omega t}2,|\xi_d-\eta_d^{\rm r}|<\frac{\omega t}2\}$, it holds 
\[
|\xi|^2=|\xi^-|^2+|\xi_d|^2<\left(\frac{\omega t}2\right)^2+\left(\frac{\omega t}2+\eta_d^{\rm r}\right)^2=\frac{\omega^2t^2}2+\omega t\eta_d^{\rm r}+(\eta_d^{\rm r})^2.
\]
Hence, for $t>t^*$,
\[
\omega^2t^2+(\eta_d^{\rm r})^2-|\xi|^2>\frac{\omega^2t^2}2-\omega t\eta_d^{\rm r}>\frac{\omega^2t^2}4,
\]
which gives 
\begin{align*}
{\rm I}_3 &\le\int_{\{\xi:|\xi^-|<\frac{\omega t}2,|\xi_d-\eta_d^{\rm r}|<\frac{\omega t}2\}}\frac{(1+|\xi|^2)^s}{||\xi|^2-\omega^2t^2-(\eta_d^{\rm r})^2|}|\widehat{\mathcal J^{-s}f}||\widehat{\mathcal J^{-s}g}|d\xi\\
&\lesssim\frac1{(\omega t)^{2-2s}}\|f\|_{H^{-s}(G)}\|g\|_{H^{-s}(G)}.
\end{align*}
We then conclude
\begin{align}\label{eq:1}
|{\rm I}|\lesssim\frac1{(\omega t)^{2-2s}}\|f\|_{H^{-s}(G)}\|g\|_{H^{-s}(G)}.
\end{align}

To estimate ${\rm II}$, we divide it into two parts
\begin{align*}
{\rm II}&=\int_{\Omega_2\cap\{\xi:|\xi_d-\eta_d^{\rm r}|>\frac{\omega t}2\}}\frac{(1+|\xi|^2)^s}{|\xi|^2-\omega^2t^2-(\eta_d^{\rm r})^2+2{\rm i}\eta_d^{\rm i}(\xi_d-\eta_d^{\rm r})}\widehat{\mathcal J^{-s}f}(\xi)\overline{\widehat{\mathcal J^{-s}g}(\xi)}d\xi\\
&\quad +\int_{\Omega_2\cap\{\xi:|\xi_d-\eta_d^{\rm r}|<\frac{\omega t}2\}}\frac{(1+|\xi|^2)^s}{|\xi|^2-\omega^2t^2-(\eta_d^{\rm r})^2+2{\rm i}\eta_d^{\rm i}(\xi_d-\eta_d^{\rm r})}\widehat{\mathcal J^{-s}f}(\xi)\overline{\widehat{\mathcal J^{-s}g}(\xi)}d\xi\\
&=:{\rm II}_1+{\rm II}_2,
\end{align*}
where ${\rm II}_1$ can be estimated similarly as ${\rm I}_2$ by utilizing the boundedness of $|\xi^-|$:
\begin{align*}
|{\rm II}_1|\lesssim\frac1{(\omega t)^{2-2s}}\|f\|_{H^{-s}(G)}\|g\|_{H^{-s}(G)}.
\end{align*}

It suffices to estimate ${\rm II}_2$ where the integrand is singular. To deal with the singularity, we denote 
\[
n_t(\xi):=\frac1{|\xi|^2-\omega^2t^2-(\eta_d^{\rm r})^2+2{\rm i}\eta_d^{\rm i}(\xi_d-\eta_d^{\rm r})}
\]
and define the transformation 
\[
\tau:\xi\mapsto\xi^*=(\xi',-\xi_d+2\eta_d^{\rm r}),\quad\xi\in\Omega_2,
\] 
where
\[
\xi':=\left(\frac{2\omega t}{|\xi^-|}-1\right)\xi^-. 
\]
A simple calculation yields that $|\xi'|=2\omega t-|\xi^-|$ and the Jacobian of the transformation is
\[
J_{d,t}(\xi)=\left|\det\frac{\partial \xi^*}{\partial \xi}\right|=\left(\frac{2\omega t}{|\xi^-|}-1\right)^{d-2}. 
\]
Moreover, it can be verified that the transformation maps the subdomain
\[
\Omega_{21}:=\left\{\xi:\frac{\omega t}2<|\xi^-|<\omega t,|\xi_d-\eta_d^{\rm r}|<\frac{\omega t}2\right\}
\]
to the subdomain
\[
\Omega_{22}:=\left\{\xi:\omega t<|\xi^-|<\frac{3\omega t}2,|\xi_d-\eta_d^{\rm r}|<\frac{\omega t}2\right\},
\]
and vice versa. 

Based on $\Omega_{21}$ and $\Omega_{22}$, ${\rm II}_2$ can be subdivided into several parts:
\begin{align*}
{\rm II}_2&=\int_{\Omega_2\cap\{\xi:|\xi_d-\eta_d^{\rm r}|<\frac{\omega t}2\}}\frac{(1+|\xi|^2)^s}{|\xi|^2-\omega^2t^2-(\eta_d^{\rm r})^2+2{\rm i}\eta_d^{\rm i}(\xi_d-\eta_d^{\rm r})}\widehat{\mathcal J^{-s}f}(\xi)\overline{\widehat{\mathcal J^{-s}g}(\xi)}d\xi\\
&=\int_{\Omega_{21}\cup\Omega_{22}}n_t(\xi)(1+|\xi|^2)^s\widehat{\mathcal J^{-s}f}(\xi)\overline{\widehat{\mathcal J^{-s}g}(\xi)}d\xi\\
&=\int_{\Omega_{22}}\left[n_t(\xi)(1+|\xi|^2)^s\widehat{\mathcal J^{-s}f}(\xi)\overline{\widehat{\mathcal J^{-s}g}(\xi)}+n_t(\xi^*)J_{d,t}(\xi)(1+|\xi^*|^2)^s\widehat{\mathcal J^{-s}f}(\xi^*)\overline{\widehat{\mathcal J^{-s}g}(\xi^*)}\right]d\xi\\
&=\int_{\Omega_{22}}\left[n_t(\xi)+n_t(\xi^*)J_{d,t}(\xi)\right](1+|\xi|^2)^s\widehat{\mathcal J^{-s}f}(\xi)\overline{\widehat{\mathcal J^{-s}g}(\xi)}d\xi\\
&\quad +\int_{\Omega_{22}}n_t(\xi^*)J_{d,t}(\xi)\left[(1+|\xi^*|^2)^s-(1+|\xi|^2)^s\right]\widehat{\mathcal J^{-s}f}(\xi)\overline{\widehat{\mathcal J^{-s}g}(\xi)}d\xi\\
&\quad +\int_{\Omega_{22}}n_t(\xi^*)J_{d,t}(\xi)(1+|\xi^*|^2)^s\left[\widehat{\mathcal J^{-s}f}(\xi^*)-\widehat{\mathcal J^{-s}f}(\xi)\right]\overline{\widehat{\mathcal J^{-s}g}(\xi)}d\xi\\
&\quad +\int_{\Omega_{22}}n_t(\xi^*)J_{d,t}(\xi)(1+|\xi^*|^2)^s\widehat{\mathcal J^{-s}f}(\xi^*)\overline{\left[\widehat{\mathcal J^{-s}g}(\xi^*)-\widehat{\mathcal J^{-s}g}(\xi)\right]}d\xi\\
&=:{\rm II}_{21}+{\rm II}_{22}+{\rm II}_{23}+{\rm II}_{24},
\end{align*}
where we used the fact
\begin{align*}
&\int_{\Omega_{21}}n_t(\xi)(1+|\xi|^2)^s\widehat{\mathcal J^{-s}f}(\xi)\overline{\widehat{\mathcal J^{-s}g}(\xi)}d\xi\\
&=\int_{\Omega_{21}}n_t(\xi^*)(1+|\xi^*|^2)^s\widehat{\mathcal J^{-s}f}(\xi^*)\overline{\widehat{\mathcal J^{-s}g}(\xi^*)}d\xi^*\\
&=\int_{\Omega_{22}}n_t(\xi^*)(1+|\xi^*|^2)^s\widehat{\mathcal J^{-s}f}(\xi^*)\overline{\widehat{\mathcal J^{-s}g}(\xi^*)}J_{d,t}(\xi)d\xi.
\end{align*}
Noting
\begin{align*}
n_t(\xi^*)&=\frac1{|\xi^*|^2-\omega^2t^2-(\eta_d^{\rm r})^2+2{\rm i}\eta_d^{\rm i}(\xi_d^*-\eta_d^{\rm r})}\\
&=\frac1{|\xi'|^2-\omega^2t^2+(\xi_d^*-\eta_d^{\rm r})(\xi_d^*+\eta_d^{\rm r})+2{\rm i}\eta_d^{\rm i}(\xi_d^*-\eta_d^{\rm r})}\\
&=\frac1{|\xi'|^2-\omega^2t^2+(\xi_d-\eta_d^{\rm r})(\xi_d-3\eta_d^{\rm r})-2{\rm i}\eta_d^{\rm i}(\xi_d-\eta_d^{\rm r})},
\end{align*}
we get for $d=2$ that
\begin{align*}
h_2(\xi):&=|n_t(\xi)+n_t(\xi^*)J_{2,t}(\xi)|\\
&=\bigg|\frac1{|\xi^-|^2-\omega^2t^2+(\xi_d-\eta_d^{\rm r})(\xi_d+\eta_d^{\rm r})+2{\rm i}\eta_d^{\rm i}(\xi_d-\eta_d^{\rm r})}\\
&\quad +\frac1{|\xi'|^2-\omega^2t^2+(\xi_d-\eta_d^{\rm r})(\xi_d-3\eta_d^{\rm r})-2{\rm i}\eta_d^{\rm i}(\xi_d-\eta_d^{\rm r})}\bigg|\\
&=\frac{2(|\xi^-|-\omega t)^2+2(\xi_d-\eta_d^{\rm r})^2}{\left[((|\xi^-|-\omega t)(|\xi^-|+\omega t)+(\xi_d-\eta_d^{\rm r})(\xi_d+\eta_d^{\rm r}))^2+4(\eta_d^{\rm i})^2(\xi_d-\eta_d^{\rm r})^2\right]^{\frac12}}\\
&\quad \times\frac1{\left[((|\xi^-|-\omega t)(|\xi^-|-3\omega t)+(\xi_d-\eta_d^{\rm r})(\xi_d-3\eta_d^{\rm r}))^2+4(\eta_d^{\rm i})^2(\xi_d-\eta_d^{\rm r})^2\right]^{\frac12}},
\end{align*}
which is bounded
\[
h_2(\xi)\lesssim\frac1{\omega^2t^2},\quad\xi\in\Omega_{22}
\]
as $t\gg1$ according to the boundedness of $\xi\in\Omega_{22}$.
Similarly, it holds for $d=3$ and $t\gg1$ that
\begin{align*}
h_3(\xi):&=|n_t(\xi)+n_t(\xi^*)J_{3,t}(\xi)|\\
&=\bigg|\frac1{|\xi^-|^2-\omega^2t^2+(\xi_d-\eta_d^{\rm r})(\xi_d+\eta_d^{\rm r})+2{\rm i}\eta_d^{\rm i}(\xi_d-\eta_d^{\rm r})}\\
&\quad +\frac{\frac{2\omega t}{|\xi^-|}-1}{|\xi'|^2-\omega^2t^2+(\xi_d-\eta_d^{\rm r})(\xi_d-3\eta_d^{\rm r})-2{\rm i}\eta_d^{\rm i}(\xi_d-\eta_d^{\rm r})}\bigg|\\
&\lesssim\frac1{\omega^2t^2}.
\end{align*}
The above estimates lead to 
\[
|{\rm II}_{21}|\lesssim\frac1{\omega^2t^2}\int_{\Omega_{22}}(1+|\xi|^2)^s|\widehat{\mathcal J^{-s}f}(\xi)||\widehat{\mathcal J^{-s}g}(\xi)|d\xi
\lesssim\frac1{(\omega t)^{2-2s}}\|f\|_{H^{-s}(G)}\|g\|_{H^{-s}(G)}.
\]
For ${\rm II}_{22}$, we apply the mean value theorem and get for some $\theta\in(0,1)$ that
\begin{align*}
&\left|n_t(\xi^*)J_{d,t}(\xi)\left[(1+|\xi^*|^2)^s-(1+|\xi|^2)^s\right]\right|\\
&=\left|n_t(\xi^*)J_{d,t}(\xi)s\left(1+\theta|\xi^*|^2+(1-\theta)|\xi|^2\right)^{s-1}(|\xi^*|^2-|\xi|^2)\right|\\
&\lesssim\left|n_t(\xi^*)J_{d,t}(\xi)(|\xi^*|^2-|\xi|^2)\right|\left(1+\theta|\xi^*|^2+(1-\theta)|\xi|^2\right)^{s-1}\\
&\lesssim\left(1+\theta|\xi^*|^2+(1-\theta)|\xi|^2\right)^{s-1}
\lesssim\frac1{(\omega t)^{2-2s}},
\end{align*}
where in the third step we used the following estimate similar to $h_2(\xi)$: 
\begin{align}\label{eq:nt}
&\left|n_t(\xi^*)J_{d,t}(\xi)(|\xi^*|^2-|\xi|^2)\right|\notag\\
&=\left|\frac{\left(\frac{2\omega t}{|\xi^-|}-1\right)^{d-2}(|\xi^*|^2-|\xi|^2)}{|\xi'|^2-\omega^2t^2+(\xi_d-\eta_d^{\rm r})(\xi_d-3\eta_d^{\rm r})-2{\rm i}\eta_d^{\rm i}(\xi_d-\eta_d^{\rm r})}\right|\notag\\
&=\frac{\left(\frac{2\omega t}{|\xi^-|}-1\right)^{d-2}\big|4\omega t(|\xi^-|-\omega t)+4\eta_d^{\rm r}(\xi_d-\eta_d^{\rm r})\big|}{\left[((|\xi^-|-\omega t)(|\xi^-|-3\omega t)+(\xi_d-\eta_d^{\rm r})(\xi_d-3\eta_d^{\rm r}))^2+4(\eta_d^{\rm i})^2(\xi_d-\eta_d^{\rm r})^2\right]^{\frac12}}
\lesssim 1.
\end{align}  
Therefore 
\[
|{\rm II}_{22}|\lesssim\frac1{(\omega t)^{2-2s}}\int_{\Omega_{22}}|\widehat{\mathcal J^{-s}f}(\xi)||\widehat{\mathcal J^{-s}g}(\xi)|d\xi
\lesssim\frac1{(\omega t)^{2-2s}}\|f\|_{H^{-s}(G)}\|g\|_{H^{-s}(G)}.
\]
Terms ${\rm II}_{23}$ and ${\rm II}_{24}$ can be estimated similarly by following the procedure used in \cite[Theorem 3.2]{LW21}. In fact, it can be shown that the Bessel potential satisfies 
\begin{align*}
\big|\widehat{\mathcal J^{-s}f}(\xi^*)-\widehat{\mathcal J^{-s}f}(\xi)\big|
\lesssim&\big||\xi^*|-|\xi|\big|\Big[M(|\nabla\widehat{\mathcal{J}^{-s}f}|)(\xi^*)+M(|\nabla\widehat{\mathcal{J}^{-s}f}|)(\xi^*)\Big],
\end{align*}
where $M$ is the Hardy--Littlewood maximal function defined by
\[
M(f)(x)=\sup_{r>0}\frac1{|B(x,r)|}\int_{B(x,r)}|f(y)|dy
\]
with $B(x,r)$ being the ball of center $x$ and radius $r$,
and satisfies (cf. \cite[Theorem 3.2]{LW21})
\[
\|M(|\nabla\widehat{\mathcal{J}^{-s}f}|)\|_{L^2(\mathbb R^d)}\lesssim\|f\|_{H^{-s}(G)}.
\]
The above estimates, together with \eqref{eq:nt}, yield
\begin{align*}
|{\rm II}_{23}|&\lesssim\int_{\Omega_{22}}\frac{|n_t(\xi^*)J_{d,t}(|\xi^*|^2-|\xi|^2)|}{|\xi^*|+|\xi|}(1+|\xi^*|^2)^s\Big|M(|\nabla\widehat{\mathcal{J}^{-s}f}|)(\xi^*)+M(|\nabla\widehat{\mathcal{J}^{-s}f}|)(\xi^*)\Big||\widehat{\mathcal J^{-s}g}(\xi)|d\xi\\
&\lesssim\frac1{(\omega t)^{1-2s}}\|f\|_{H^{-s}(G)}\|g\|_{H^{-s}(G)}
\end{align*}
and
\begin{align*}
|{\rm II}_{24}|&\lesssim\int_{\Omega_{22}}\frac{|n_t(\xi^*)J_{d,t}(|\xi^*|^2-|\xi|^2)|}{|\xi^*|+|\xi|}(1+|\xi^*|^2)^s|\widehat{\mathcal J^{-s}f}(\xi^*)|\Big|M(|\nabla\widehat{\mathcal{J}^{-s}g}|)(\xi^*)+M(|\nabla\widehat{\mathcal{J}^{-s}g}|)(\xi^*)\Big|d\xi\\
&\lesssim\frac1{(\omega t)^{1-2s}}\|f\|_{H^{-s}(G)}\|g\|_{H^{-s}(G)}.
\end{align*}
Hence, ${\rm II}$ satisfies
\begin{align}\label{eq:2}
|{\rm II}|\lesssim\frac1{(\omega t)^{1-2s}}\|f\|_{H^{-s}(G)}\|g\|_{H^{-s}(G)}.
\end{align}

Combining \eqref{eq:1} and \eqref{eq:2}, we obtain the estimate of $\mathcal A$ and get 
\[
|\langle\mathcal G_\eta f,g\rangle|\lesssim\frac1{\omega^{3-2s}t^{1-2s}}\|f\|_{H^{-s}(G)}\|g\|_{H^{-s}(G)}
\]
for any $f,g\in C_0^\infty(G)$. Since  $C_0^\infty(G)$ is dense in $L^2(G)$ and $H^{-s}(G)\subset H^{-1}(G)=\overline{L^2(G)}^{\|\cdot\|_{H^{-1}(G)}}$ (cf. \cite[Sections 2.30, 3.13]{AF03}), the above result can be extended to $f,g\in H^{-s}(G)$ with $s\in(0,\frac12)$. Therefore we derive the following estimate for the operator $\mathcal G_\eta$ with $s\in(0,\frac12)$:
\begin{align}\label{eq:G1}
\|\mathcal G_\eta\|_{\mathcal L(H^{-s}(G),H^s(G))}\lesssim\frac1{\omega^{3-2s}t^{1-2s}}.
\end{align}

The second step is to estimate the operator $\mathcal G_\eta$ in Sobolev spaces and show $v\equiv0$ in $\mathbb R^d$. To extend the estimate of $\mathcal G_\eta$ from Hilbert spaces to Sobolev spaces, we claim that $\mathcal G_\eta:L^r(G)\to L^{r'}(G)$ is bounded and satisfies
\begin{align}\label{eq:G2}
\|\mathcal G_\eta\|_{\mathcal L(L^{r}(G),L^{r'}(G))}\lesssim1
\end{align}
for some proper $r$ and $r'$. In fact, it follows from the decomposition of the operator $\mathcal G_\eta$ given in \eqref{eq:Geta} that we may rewrite it as
\[
\mathcal G_\eta=\frac1{2\kappa^2}\left(\mathcal G_{\eta,1}-\mathcal G_{\eta,2}\right),
\]
where 
\[
\mathcal G_{\eta,1}(f)(x):=\mathcal F^{-1}\bigg[\frac{\hat{f}}{|\xi|^2+2\eta\cdot\xi}\bigg](x),\quad
\mathcal G_{\eta,2}(f)(x):=\mathcal F^{-1}\bigg[\frac{\hat{f}}{|\xi|^2+2\eta\cdot\xi+2\kappa^2}\bigg](x).
\]

Next we consider the cases $d=3$ and $d=2$, separately. 

For $d=3$, the claim \eqref{eq:G2} holds under the conditions
\[
\frac1r-\frac1{r'}=\frac2d,\quad\min\left\{\left|\frac1r-\frac12\right|,\left|\frac1{r'}-\frac12\right|\right\}>\frac1{2d},
\]
since operators $\mathcal G_{\eta,i}$, $i=1,2$, are both bounded from $L^r(G)$ to $L^{r'}(G)$ according to \cite[Theorem 2.2]{KRS87} and \cite[Proposition 2]{LPS08}. To deduce the estimate for $\mathcal G_\eta$ between the dual Sobolev spaces $W^{-\gamma,p}(G)$ and $W^{\gamma,q}(G)$ with $\frac1p+\frac1q=1$, we consider the interpolation of \eqref{eq:G1} and \eqref{eq:G2}. Noting
\begin{align*}
&[L^r(G),H^{-s}(G)]_\theta=W^{-\gamma,p}(G),\\
&[L^{r'}(G),H^s(G)]_\theta=W^{\gamma,q}(G)
\end{align*}
and choosing $\theta=1+(\frac1q-\frac12)d\in(0,1)$ and $r=\frac{2d}{d+2}$ such that $\gamma=\theta s<\frac12+(\frac1q-\frac12)\frac d2$, $\frac1p=\frac{1-\theta}r+\frac{\theta}{2}$ and $\frac1q=\frac{1-\theta}{r'}+\frac{\theta}2$, we obtain 
\begin{align}\label{eq:G3}
\|\mathcal G_\eta\|_{\mathcal L(W^{-\gamma,p}(G),W^{\gamma,q}(G))}\lesssim\frac1{\omega^{(3-2s)\theta}t^{(1-2s)\theta}}.
\end{align}
As is proved in \cite[Lemma 2]{LPS08}, $\rho v\in W^{-\gamma,p}(G)$ for any $v\in W^{\gamma,q}(G)$, where $\gamma$ is required to satisfy $\gamma<\frac12+(\frac1q-\frac12)\frac d2$. Hence an additional restriction on $q$ is also required due to $\gamma>\frac{d-m}2$, i.e., $q<\frac{2d}{3d-2m-2}$. 
Consequently, \eqref{eq:v} leads to
\[
\|v\|_{W^{\gamma,q}(G)}\le\|\mathcal G_\eta\|_{\mathcal L(W^{-\gamma,p}(G),W^{\gamma,q}(G))}\|\rho v\|_{W^{-\gamma,p}(G)}
\lesssim\frac1{\omega^{(3-2s)\theta}t^{(1-2s)\theta}}\|v\|_{W^{\gamma,q}(G)}
\]
with $s\in(0,\frac12)$, which implies $v\equiv0$ by choosing $t\gg1$.

For $d=2$, it is shown in \cite[Proposition 2]{LPS08} that \eqref{eq:G2} holds for any $r>1$. Similarly, \eqref{eq:G3} can be deduced from the interpolation between \eqref{eq:G1} and \eqref{eq:G2} by choosing $r=1+\epsilon$ with an arbitrary small parameter $\epsilon>0$ and $\theta=\frac{2(1+\epsilon)-2\epsilon q}{q(1-\epsilon)}$ such that $\gamma=\theta s<\frac{(1+\epsilon)-\epsilon q}{q(1-\epsilon)}$. Following the same procedure as the three-dimensional case and letting $\epsilon\to0$, we get $v\equiv0$ under the restrictions $\gamma<\frac1q=\frac12+(\frac1q-\frac12)\frac d2$ and $q<\frac2{2-m}=\frac{2d}{3d-2m-2}$.
\end{proof}

\begin{remark}
The unique continuation principle established in Theorem \ref{tm:unique} holds for any damping coefficient $\sigma\ge0$. If the medium is lossless with $\sigma=0$, the proof can be simplified by letting $\omega=k^{\frac12}$ and
\[
\eta=\left(k^{\frac12}t,0,\cdots,0,{\rm i}k^{\frac12}\sqrt{t^2-1}\right)^\top,\quad t\gg1. 
\] 
We refer to \cite{P04} for the unique continuation principle of the Schr\"odinger equation without damping. The unique continuation principle will be utilized to show the uniqueness of the solution to the direct scattering problem when $\sigma=0$ .  
\end{remark}

\section{The Lippmann--Schwinger equation}
\label{sec:4}

In this section, we examine the well-posedness of the scattering problem \eqref{eq:model}--\eqref{eq:radiation} by studying the equivalent Lippmann--Schwinger integral equation.

\subsection{Well-posedness}

Based on the integral operators, the scattering problem \eqref{eq:model}--\eqref{eq:radiation} can be written formally as the Lippmann--Schwinger equation
\begin{align}\label{eq:LS}
u=\mathcal K_ku+\mathcal H_k\delta_y=\mathcal K_ku+\Phi,
\end{align}
where the fundamental solution $\Phi$ is given in \eqref{eq:Phi}.

\begin{theorem}\label{tm:LS}
Let the random potential $\rho$ satisfy Assumption \ref{as:rho}. The Lippmann--Schwinger equation \eqref{eq:LS} has a unique solution in $W^{\gamma,q}_{loc}(\mathbb R^d)$ with $q\in(2,\frac{2d}{3d-2m-2})$ and $\gamma\in(\frac{d-m}2,\frac12+(\frac1q-\frac12)\frac d2)$.
\end{theorem}
\begin{proof}
According to the compactness of the operator $\mathcal K_k$ proved in Lemma \ref{lm:operatorK} and the Fredholm alternative theorem, it suffices to show that the homogeneous equation
\begin{align}\label{eq:homo}
u=\mathcal K_k u
\end{align}
has only the trivial solution $u\equiv0$. 

Assume that $u^*$ is a solution to the homogeneous equation \eqref{eq:homo}. Then it satisfies the following equation in the distribution sense: 
\begin{equation}\label{eq:homor}
\Delta^2u^*-\kappa^4u^*+\rho u^*=0\quad\text{in}~\mathbb R^d.
\end{equation}
Let us consider two auxiliary functions 
\[
u_H:=-\frac1{2\kappa^2}(\Delta u^*-\kappa^2u^*),\quad u_M:=\frac1{2\kappa^2}(\Delta u^*+\kappa^2u^*).
\]
It is clear to note that $u^*=u_H+u_M$ and $\Delta u^*=\kappa^2(u_M-u_H)$.

Since $\rho$ is compactly supported in $D$, there exists a constant $R>0$ such that $D\subset B_R$ with $B_R$ being the open ball of radius $R$ centered at zero. It can be verified that $u_H$ and $u_M$ satisfy the homogeneous Helmholtz and modified Helmholtz equation with the wavenumber $\kappa$, respectively, in $\mathbb R^d\setminus\overline{B_R}$:
\[
\Delta u_H+\kappa^2u_H=0,\quad\Delta u_M-\kappa^2 u_M=0.
\]
Hence, $u_H$ and $u_M$ admit the following Fourier series expansions for any $r=|x|>R$:
\begin{align}\label{eq:expan1}
\begin{aligned}
u_H(r,\theta)&=\sum_{n=-\infty}^{\infty}\frac{H_n^{(1)}(\kappa r)}{H_n^{(1)}(\kappa R)}\hat u^{(n)}_H(R)e^{{\rm i}n\theta},\\
u_M(r,\theta)&=\sum_{n=-\infty}^{\infty}\frac{K_n(\kappa r)}{K_n(\kappa R)}\hat u^{(n)}_M(R)e^{{\rm i}n\theta}, 
\end{aligned}
\quad\text{if}\quad d=2,
\end{align}
where 
\[
\hat u^{(n)}_J(R)=\frac1{2\pi}\int_0^{2\pi}u_J(R,\theta)e^{-{\rm i}n\theta}d\theta, \quad J\in\{H, M\} 
\]
are the Fourier coefficients, and
\begin{align}\label{eq:expan2}
\begin{aligned}
u_H(r,\theta,\varphi)&=\sum_{n=0}^\infty\sum_{m=-n}^n\frac{h_n^{(1)}(\kappa r)}{h_n^{(1)}(\kappa R)}\hat u^{(m,n)}_{H}(R)Y_n^m(\theta,\varphi),\\
u_M(r,\theta,\varphi)&=\sum_{n=0}^\infty\sum_{m=-n}^n\frac{k_n(\kappa r)}{k_n(\kappa R)}\hat u_{M}^{(m,n)}(R)Y_n^m(\theta,\varphi),
\end{aligned}
\quad\text{if}\quad d=3,
\end{align}
where $h_n^{(1)}$ and $k_n$ are the spherical and modified spherical Hankel functions, respectively, satisfying
\[
h_n^{(1)}(z)=\sqrt{\frac{\pi}{2z}}H_{n+\frac12}^{(1)}(z),\quad k_n(z)=\sqrt{\frac{\pi}{2z}}K_{n+\frac12}(z),\quad z\in\mathbb C,
\]
$Y_n^m$ are the spherical harmonics of order $n$, and the Fourier coefficients $\hat u_{J}^{(m,n)}(R)$ are given by 
\[
\hat u_{J}^{(m,n)}(R)=\int_{\mathbb S^2}u_J(R,\theta,\varphi)\overline{Y_n^m(\theta,\varphi)}ds.
\]

If $\sigma>0$, then we have $\kappa_{\rm r}=\Re(\kappa)>0,\kappa_{\rm i}=\Im(\kappa)>0$. It follows from \eqref{eq:Hv}--\eqref{eq:Kv} and \eqref{eq:expan1}--\eqref{eq:expan2} that $u_H, u_M$ and thus $u^*, \Delta u^*$ decay exponentially as $r\to\infty$. Multiplying \eqref{eq:homor} by the complex conjugate of $u^*$, integrating over $B_r$, and applying Green's formula, we obtain 
\begin{align*}
\int_{B_r}\left(|\Delta u^*|^2-\kappa^4|u^*|^2+\rho|u^*|^2\right)dx=&\int_{\partial B_r}\big(\Delta u^* \overline{\partial_\nu u^*}-\overline{u^*}\partial_\nu\Delta u^*\big)ds,
\end{align*}
where $\nu$ is the unit outward normal vector to $\partial B_r$. Taking the imaginary part of the above equation yields
\[
-\Im(\kappa^4)\|u^*\|_{L^2(B_r)}^2=\Im\left[\int_{\partial B_r}\big(\Delta u^* \overline{\partial_\nu u^*}-\overline{u^*}\partial_\nu\Delta u^*\big)ds\right]\to0
\]
as $r\to\infty$ and hence $u^*\equiv0$ in $\mathbb R^d$.

If $\sigma=0$, then $\kappa=k^\frac12$ is real. By \eqref{eq:expan1}--\eqref{eq:expan2}, only $u_M|_{\partial B_r}$ and $\partial_\nu u_M|_{\partial B_r}$ decay exponentially as $r\to\infty$. It is easy to verify from \eqref{eq:homor} that $u_H$ and $u_M$ satisfy the following equations in $\mathbb R^d$: 
\[
 \Delta u_H+ku_H-\frac1{2k}\rho u^*=0,\quad\Delta u_M-ku_M+\frac1{2k}\rho u^*=0.
\]
Using the integration by parts and $u^*=u_H+u_M$, we have from Green's formula that 
\begin{align*}
\int_{\partial B_r}u_M \overline{\partial_\nu u_M}ds&=\int_{B_r}\Big(|\nabla u_M|^2+k|u_M|^2-\frac1{2k}\rho|u_M|^2-\frac1{2k}\rho u_M\overline{u_H}\Big)dx,\\
\int_{\partial B_r}u_H\overline{\partial_\nu u_H}ds&=\int_{B_r}\Big(|\nabla u_H|^2-k|u_H|^2+\frac1{2k}\rho|u_H|^2+\frac1{2k}\rho\overline{u_M}u_H\Big)dx,
\end{align*}
which are well-defined since $\nabla\Delta u^*\in L^2_{loc}(\mathbb R^d)$ due to $\Delta^2 u^*=k^2u^*-\rho u^*\in W^{\gamma,q}_{loc}(\mathbb R^d)+W^{-\gamma,p}(D)$ and \eqref{eq:rhophi}. Taking the imaginary parts of the above two equations yields 
\[
\Im\left[\int_{\partial B_r}u_M\overline{\partial_\nu u_M}ds\right]=\Im\left[\int_{\partial B_r}u_H\overline{\partial_\nu u_H}ds\right],
\]
which leads to
\begin{align*}
\int_{\partial B_r}\big(\left|\partial_\nu u_H\right|^2+k|u_H|^2\big)ds=\int_{\partial B_r}\left|\partial_\nu u_H-{\rm i}k^{\frac12}u_H\right|^2ds-2k^{\frac12}\Im\left[\int_{\partial B_r}u_M\overline{\partial_\nu u_M}ds\right].
\end{align*}
By the Sommerfeld radiation condition \eqref{eq:radiation}, the first integral on the right-hand side of the above equation tends to zero as $r\to\infty$. The second integral also tends to zero due to the exponential decay of $u_M$. Therefore,
\[
\lim_{r\to\infty}\int_{\partial B_r}\big(\left|\partial_\nu u_H\right|^2+k|u_H|^2\big)ds=\lim_{r\to\infty}\int_{\partial B_r}\big(\left|\partial_\nu u_M\right|^2+k|u_M|^2\big)ds=0.
\]
It follows from Rellich's lemma that $u_H=u_M=0$ in $\mathbb R^d\backslash\overline{B_R}$ and thus $u^*\equiv0$ in $\mathbb R^d\backslash\overline{B_R}$. The proof is completed by applying the unique continuation in Theorem \ref{tm:unique}.
\end{proof}

The well-posedness of the scattering problem \eqref{eq:model}--\eqref{eq:radiation} can be obtained by showing the equivalence to the Lippmann--Schwinger equation. The proof is similar to that of \cite[Theorem 3.5]{LW21} and is omitted here for brevity. 

\begin{corollary}
Under Assumption \ref{as:rho}, the scattering problem \eqref{eq:model}--\eqref{eq:radiation} is well-posed in the distribution sense and has a unique solution $u\in W^{\gamma,q}_{loc}(\mathbb R^d)$, where $q$ and $\gamma$ are given in Theorem \ref{tm:LS}.
\end{corollary}

\subsection{Born series}

Based on the Lippmann--Schwinger equation \eqref{eq:LS}, we formally define the Born series
\[
\sum_{n=0}^\infty u_n(x,y,k), 
\]
where
\begin{equation}\label{borns}
u_{n}(x,y,k):=\mathcal K_k\left(u_{n-1}(\cdot,y,k)\right)(x)=\int_{\mathbb R^d}\Phi(x,z,k)\rho(z)u_{n-1}(z,y,k)dz,\quad n\ge1
\end{equation}
and $u_0(x,y,k):=\mathcal H_k(\delta_y )(x)=\Phi(x,y,k)$.

The Born series is crucial for the inverse scattering problem. It helps to establish the recovery formula for the strength $\mu$ of the random potential $\rho$. Before addressing the inverse problem, we study the convergence of the Born series.

\begin{lemma}\label{lm:Born}
There exists $k_0>0$ such that for any wavenumber $k\ge k_0$ and any fixed $x,y\in U$, the Born series converges to the solution of \eqref{eq:model}--\eqref{eq:radiation}, i.e.,
\[
u(x,y,k)=\sum_{n=0}^\infty u_n(x,y,k).
\]
\end{lemma}

\begin{proof}
The convergence of the Born series to the solution of \eqref{eq:model}--\eqref{eq:radiation} can be obtained by employing the same procedure as that in \cite[Section 4.2]{LLW} and the estimate of $u_0(x,y,k)=\Phi(x,y,k)$ given in Lemma \ref{lm:Phi}.

Moreover, the Born series admits the pointwise convergence. Using the estimates of $\mathcal H_k$ and $\mathcal K_k$ given in Lemmas \ref{lm:operatorH} and \ref{lm:operatorK}, we get for any $s\in(\frac{d-m}2,\frac32)$ that
\begin{align}\label{eq:resi}
&\Big\|u(\cdot,y,k)-\sum_{n=0}^Nu_{n}(\cdot,y,k)\Big\|_{L^\infty(U)}
\lesssim\sum_{n=N+1}^\infty\|\mathcal K_k^n\left(u_0(\cdot,y,k)\right)\|_{L^\infty(U)}\notag\\
&\lesssim\sum_{n=N+1}^\infty\|\mathcal K_k\|_{\mathcal L(H^s(U),L^\infty(U))}\|\mathcal K_k\|_{\mathcal L(H^s(U))}^{n-2}\|\mathcal H_k\|_{\mathcal L(H^{-s}(D),H^s(U))}\|\rho\Phi(\cdot,y,k)\|_{H^{-s}(D)}\notag\\
&\lesssim\sum_{n=N+1}^\infty k^{\frac{2s+d-6+\epsilon}4}k^{(s-\frac32)(n-2)}k^{s-\frac32}\|\Phi(\cdot,y,k)\|_{H^{s}(D)}\notag\\
&\lesssim k^{\frac{2s+d-6+\epsilon}4+(s-\frac32)N+\frac{d-7}4+\frac s2}\to0
\end{align}
as $N\to\infty$ for any $k\ge k_0$ and $\epsilon>0$, where we used \eqref{eq:K} and Lemma \ref{lm:Phi}.
\end{proof}

\section{The inverse scattering problem}
\label{sec:inverse}

This section is devoted to the inverse scattering problem, which is to determine the strength $\mu$ of the random potential $\rho$. More specifically, the point source is assumed to be located at $y=x$, where $x\in U$ is the observation point and $U$ is the measurement domain having a positive distance to the support $D$ of the random potential. Hence only the backscattering data is utilized for the inverse problem. For simplicity, we use the notation $u_n(x,k):=u_n(x,x,k)$. The scattered field, denoted by $u^s$, has the form
\[
u^s(x,k)=u(x,k)-u_0(x,k)=\sum_{n=1}^\infty u_n(x,k)
\]
for $k\ge k_0$ with $k_0$ being given in Lemma \ref{lm:Born}.

Next we analyze the contribution of each term in the Born series in order to deduce the reconstruction formula and achieve the uniqueness of the inverse problem.

\subsection{The analysis of $u_1$}
\label{sec:5.1}

Based on the definitions of the Born sequence \eqref{borns} and the incident field $u_0$, the leading term $u_1$ can be expressed as
\begin{equation}\label{u1}
u_1(x,k)=\mathcal K_k(u_0(\cdot,x,k))(x)=\int_{\mathbb R^d}\Phi(x,z,k)^2\rho(z)dz.
\end{equation}
Since the fundamental solutions take different forms, the contribution of $u_1$ is discussed for the three- and two-dimensional cases, separately.

\subsubsection{The three-dimensional case}

By Assumption \ref{as:rho}, we have $m\in(2,3]$ for $d=3$. Substituting the fundamental solution 
\[
\Phi(x,z,k)=-\frac1{8\pi \kappa^2|x-z|}\big(e^{{\rm i}\kappa|x-z|}-e^{-\kappa|x-z|}\big)
\] 
into \eqref{u1} gives 
\begin{align*}
\mathbb E|u_1(x,k)|^2&=\frac{1}{(8\pi|\kappa|^2)^4}\int_{\mathbb R^3}\int_{\mathbb R^3}\left(\frac{e^{{\rm i}\kappa|x-z|}-e^{-\kappa|x-z|}}{|x-z|}\right)^2\left(\overline{\frac{e^{{\rm i}\kappa|x-z'|}-e^{-\kappa|x-z'|}}{|x-z'|}}\right)^2\mathbb E[\rho(z)\rho(z')]dzdz'\\
&=\frac{1}{(8\pi|\kappa|^2)^4}\int_{\mathbb R^3}\int_{\mathbb R^3}\frac{e^{2{\rm i}(\kappa|x-z|-\overline{\kappa}|x-z'|)}}{|x-z|^2|x-z'|^2}\mathbb E[\rho(z)\rho(z')]dzdz'\\
&\quad -\frac{2}{(8\pi|\kappa|^2)^4}\int_{\mathbb R^3}\int_{\mathbb R^3}\frac{e^{2{\rm i}\kappa|x-z|-({\rm i}+1)\overline{\kappa}|x-z'|}}{|x-z|^2|x-z'|^2}\mathbb E[\rho(z)\rho(z')]dzdz'\\
&\quad +\frac{1}{(8\pi|\kappa|^2)^4}\int_{\mathbb R^3}\int_{\mathbb R^3}\frac{e^{2{\rm i}\kappa|x-z|-2\overline{\kappa}|x-z'|}}{|x-z|^2|x-z'|^2}\mathbb E[\rho(z)\rho(z')]dzdz'\\
&\quad -\frac{2}{(8\pi|\kappa|^2)^4}\int_{\mathbb R^3}\int_{\mathbb R^3}\frac{e^{({\rm i}-1)\kappa|x-z|-2{\rm i}\overline{\kappa}|x-z'|}}{|x-z|^2|x-z'|^2}\mathbb E[\rho(z)\rho(z')]dzdz'\\
&\quad +\frac{4}{(8\pi|\kappa|^2)^4}\int_{\mathbb R^3}\int_{\mathbb R^3}\frac{e^{({\rm i}-1)\kappa|x-z|-({\rm i}+1)\overline{\kappa}|x-z'|}}{|x-z|^2|x-z'|^2}\mathbb E[\rho(z)\rho(z')]dzdz'\\
&\quad -\frac{2}{(8\pi|\kappa|^2)^4}\int_{\mathbb R^3}\int_{\mathbb R^3}\frac{e^{({\rm i}-1)\kappa|x-z|-2\overline{\kappa}|x-z'|}}{|x-z|^2|x-z'|^2}\mathbb E[\rho(z)\rho(z')]dzdz'\\
&\quad +\frac{1}{(8\pi|\kappa|^2)^4}\int_{\mathbb R^3}\int_{\mathbb R^3}\frac{e^{-2\kappa|x-z|-2{\rm i}\overline{\kappa}|x-z'|}}{|x-z|^2|x-z'|^2}\mathbb E[\rho(z)\rho(z')]dzdz'\\
&\quad -\frac{2}{(8\pi|\kappa|^2)^4}\int_{\mathbb R^3}\int_{\mathbb R^3}\frac{e^{-2\kappa|x-z|-({\rm i}+1)\overline{\kappa}|x-z'|}}{|x-z|^2|x-z'|^2}\mathbb E[\rho(z)\rho(z')]dzdz'\\
&\quad +\frac{1}{(8\pi|\kappa|^2)^4}\int_{\mathbb R^3}\int_{\mathbb R^3}\frac{e^{-2\kappa|x-z|-2\overline{\kappa}|x-z'|}}{|x-z|^2|x-z'|^2}\mathbb E[\rho(z)\rho(z')]dzdz'\\
&=\frac{1}{(8\pi|\kappa|^2)^4}\int_{\mathbb R^3}\int_{\mathbb R^3}\frac{e^{2{\rm i}(\kappa|x-z|-\overline{\kappa}|x-z'|)}}{|x-z|^2|x-z'|^2}\mathbb E[\rho(z)\rho(z')]dzdz'\\
&\quad -\frac{4}{(8\pi|\kappa|^2)^4}\Re\int_{\mathbb R^3}\int_{\mathbb R^3}\frac{e^{2{\rm i}\kappa|x-z|-({\rm i}+1)\overline{\kappa}|x-z'|}}{|x-z|^2|x-z'|^2}\mathbb E[\rho(z)\rho(z')]dzdz'\\
&\quad +\frac{2}{(8\pi|\kappa|^2)^4}\Re\int_{\mathbb R^3}\int_{\mathbb R^3}\frac{e^{2{\rm i}\kappa|x-z|-2\overline{\kappa}|x-z'|}}{|x-z|^2|x-z'|^2}\mathbb E[\rho(z)\rho(z')]dzdz'\\
&\quad +\frac{4}{(8\pi|\kappa|^2)^4}\int_{\mathbb R^3}\int_{\mathbb R^3}\frac{e^{({\rm i}-1)\kappa|x-z|-({\rm i}+1)\overline{\kappa}|x-z'|}}{|x-z|^2|x-z'|^2}\mathbb E[\rho(z)\rho(z')]dzdz'\\
&\quad -\frac{4}{(8\pi|\kappa|^2)^4}\Re\int_{\mathbb R^3}\int_{\mathbb R^3}\frac{e^{({\rm i}-1)\kappa|x-z|-2\overline{\kappa}|x-z'|}}{|x-z|^2|x-z'|^2}\mathbb E[\rho(z)\rho(z')]dzdz'\\
&\quad +\frac{1}{(8\pi|\kappa|^2)^4}\int_{\mathbb R^3}\int_{\mathbb R^3}\frac{e^{-2\kappa|x-z|-2\overline{\kappa}|x-z'|}}{|x-z|^2|x-z'|^2}\mathbb E[\rho(z)\rho(z')]dzdz'\\
&=:{\rm I}_1+{\rm I}_2+{\rm I}_3+{\rm I}_4+{\rm I}_5.
\end{align*}

For ${\rm I}_1$, following the procedure used in \cite[Theorem 4.5]{LW}, we get
\begin{align*}
|{\rm I}_1|&=\frac{1}{(8\pi|\kappa|^2)^4}\bigg[\int_{D}\frac{e^{-4\kappa_{\rm i}|x-z|}}{|x-z|^4}\mu(z)dz\kappa_{\rm r}^{-m}+O\left(\kappa_{\rm r}^{-m-1}\right)\bigg]\\
&=\frac{\kappa_{\rm r}^{-m}}{(8\pi|\kappa|^2)^4}\int_{D}\frac{e^{-4\kappa_{\rm i}|x-z|}}{|x-z|^4}\mu(z)dz+O\left(\kappa_{\rm r}^{-m-9}\right). 
\end{align*} 

The other terms can be estimated by utilizing the exponential decay of the integrants with respect to $\kappa_{\rm r}$. 
Since the estimates are analogous, we only show the detail for ${\rm I}_2$. Note that $|x-z|$ is bounded below and above for any $x\in U$ and $z\in D$. A simple calculation yields 
\begin{align*}
{\rm I}_2=\frac{4}{(8\pi|\kappa|^2)^4}\Re\int_{D}\int_{D}\frac{e^{{\rm i}(2\kappa_{\rm r}|x-z|+(\kappa_{\rm i}-\kappa_{\rm r})|x-z'|)}e^{-2\kappa_{\rm i}|x-z|-(\kappa_{\rm r}+\kappa_{\rm i})|x-z'|}}{|x-z|^2|x-z'|^2}\mathbb E[\rho(z)\rho(z')]dzdz',
\end{align*}
where
\[
e^{-2\kappa_{\rm i}|x-z|-(\kappa_{\rm r}+\kappa_{\rm i})|x-z'|}\lesssim\kappa_{\rm r}^{-M}
\]
for any $M>0$ as $\kappa_{\rm r}\to\infty$. Choosing $M=m+1$ gives 
\begin{align*}
|{\rm I}_2|\lesssim|\kappa|^{-8}\kappa_{\rm r}^{-m-1}\int_D\int_D|\mathbb E[\rho(z)\rho(z')]|dzdz'\lesssim\kappa_{\rm r}^{-m-9}\quad\forall~x\in U,
\end{align*}
where we used the equivalence between $|\kappa|$ and $\kappa_{\rm r}$ as $\kappa_{\rm r}\to\infty$ and the following expression (up to a constant) of the leading term for the kernel $\mathbb E[\rho(z)\rho(z')]$ (cf. \cite[Lemma 2.4]{LW2}) with $d=2,3$:
\begin{equation}\label{eq:kernel}
\mathbb E[\rho(z)\rho(z')]\sim\left\{
\begin{aligned}
&\mu(z)\ln|z-z'|,\quad &&m=d,\\
&\mu(z)|z-z'|^{m-d},\quad &&m\in(d-1,d).
\end{aligned}
\right.
\end{equation}
Terms ${\rm I}_3,{\rm I}_4$ and ${\rm I}_5$ can be estimated similarly. Hence we obtain 
\begin{align}\label{eq:u1_3d}
\mathbb E|u_1(x,k)|^2=\frac{\kappa_{\rm r}^{-m}}{(8\pi|\kappa|^2)^4}\int_{D}\frac{e^{-4\kappa_{\rm i}|x-z|}}{|x-z|^4}\mu(z)dz+O\left(\kappa_{\rm r}^{-m-9}\right)\quad\forall~x\in U.
\end{align}

\subsubsection{The two-dimensional case}

Now let us consider the two-dimensional problem where $d=2$ and $m\in(1,2]$. The fundamental solution $\Phi$ has the asymptotic expansion (cf. \cite{AS92,LW2})
\[
\Phi(x,z,k)=-\sum_{j=0}^\infty\frac{C_j}{8\kappa^2(\kappa|x-z|)^{j+\frac12}}\big({\rm i}e^{{\rm i}\kappa|x-z|}-{\rm i}^{-j+\frac12}e^{-\kappa|x-z|}\big), 
\]
where $C_0=1$ and 
\[
C_j=\sqrt{\frac2\pi}\frac{8^{-j}}{j!}\prod\limits_{l=1}^j(2l-1)^2e^{-\frac{\rm i\pi}4},\quad j\ge1.
\] 

Define the truncations of $\Phi$ and $u_1$, respectively, as follows
\begin{align*}
\Phi_N(x,z,k):&=-\sum_{j=0}^N\frac{C_j}{8\kappa^2(\kappa|x-z|)^{j+\frac12}}\big({\rm i}e^{{\rm i}\kappa|x-z|}-{\rm i}^{-j+\frac12}e^{-\kappa|x-z|}\big),\\
u_1^{(N)}(x,k):&=\int_{\mathbb R^2}\Phi_N(x,z,k)^2\rho(z)dz,
\end{align*}
where 
\[
|\Phi(x,z,k)|\lesssim|\kappa|^{-\frac52}|x-z|^{-\frac12},\quad|\Phi_N(x,z,k)|\lesssim|\kappa|^{-\frac52}|x-z|^{-\frac12}
\]
and
\begin{align}\label{eq:PhiN}
\Phi(x,z,k)-\Phi_N(x,z,k)=O\big(|\kappa|^{-N-\frac72}|x-z|^{-N-\frac32}\big)
\end{align}
for any $N\in\mathbb N$ as $|\kappa||x-z|\to\infty$. The following lemma gives the truncation error of the fundamental solution. 

\begin{lemma}\label{lm:trun}
For any fixed $x\in U$, $N\in\mathbb N$, $\gamma\in[0,1]$ and $q>1$, it holds
\begin{align}\label{eq:trun1}
\|\Phi(x,\cdot,k)-\Phi_N(x,\cdot,k)\|_{W^{\gamma,q}(D)}\lesssim|\kappa|^{-N-\frac72+\gamma}.
\end{align}
In particular, for $N=0$ and $\tilde q\in(1,\frac43)$, it holds
\begin{align}\label{eq:trun2}
\|\Phi(\cdot,\cdot,k)-\Phi_0(\cdot,\cdot,k)\|_{W^{\gamma,\tilde q}(D\times D)}\lesssim|\kappa|^{-\frac72+\gamma}.
\end{align}
\end{lemma}

\begin{proof}
Using \eqref{eq:PhiN} and 
\[
\left|\nabla_z\left(\Phi(x,z,k)-\Phi_N(x,z,k)\right)\right|=O\big(|\kappa|^{-N-\frac52}|x-z|^{-N-\frac32}\big),
\]
we get 
\begin{align*}
&\|\Phi(x,\cdot,k)-\Phi_N(x,\cdot,k)\|_{L^q(D)} \lesssim|\kappa|^{-N-\frac72},\\
&\|\Phi(x,\cdot,k)-\Phi_N(x,\cdot,k)\|_{W^{1,q}(D)} \lesssim|\kappa|^{-N-\frac52}.
\end{align*}
Then \eqref{eq:trun1} follows from the space interpolation $[L^q(D),W^{1,q}(D)]_\gamma=W^{\gamma,q}(D)$.

Similarly, \eqref{eq:trun2} can be obtained by noting that
\[
\|\Phi(\cdot,\cdot,k)-\Phi_0(\cdot,\cdot,k)\|_{L^{\tilde q}(D\times D)}\lesssim|\kappa|^{-\frac72}\Big(\int_D\int_D|z-z'|^{-\frac32\tilde q}dzdz'\Big)^{\frac1{\tilde q}}\lesssim|\kappa|^{-\frac72}
\]
and
\[
\|\Phi(\cdot,\cdot,k)-\Phi_0(\cdot,\cdot,k)\|_{W^{1,\tilde q}(D\times D)}\lesssim|\kappa|^{-\frac52}
\]
for any $\tilde q\in(1,\frac43)$.
\end{proof}

Choosing $N=1$ and using \eqref{eq:Phiesti}, \eqref{eq:kernel}, and \eqref{eq:PhiN}, we get for any $x\in U$ that
\begin{align*}
&\mathbb E\left|u_1(x,k)-u_1^{(1)}(x,k)\right|^2=\int_{D}\int_D
(\Phi^2-\Phi_1^2)(x,z,k)\overline{(\Phi^2-\Phi_1^2)(x,z',k)}\mathbb E[\rho(z)\rho(z')]dzdz'\\
&\lesssim\sup_{(x,z)\in U\times D}\left[|(\Phi+\Phi_1)(x,z,k)|^2|(\Phi-\Phi_1)(x,z,k)|^2\right]\int_D\int_D|\mathbb E[\rho(z)\rho(z')]|dzdz'\\
&\lesssim~|\kappa|^{-14}.
\end{align*}

The second moment of $u_1^{(1)}$ satisfies
\begin{align*}
\mathbb E|u_1^{(1)}(x,k)|^2 &=\frac{1}{(8|\kappa|^2)^4}\sum_{j,l=0}^1\frac{C_j^2\overline{C_l^2}}{\kappa^{2j+1}\overline{\kappa}^{2l+1}}\int_D\int_D\bigg(\frac{{\rm i}e^{{\rm i}\kappa|x-z|}-{\rm i}^{-j+\frac12}e^{-\kappa|x-z|}}{|x-z|^{j+\frac12}}\bigg)^2\\
&\quad \times\Bigg(\overline{\frac{{\rm i}e^{{\rm i}\kappa|x-z'|}-{\rm i}^{-l+\frac12}e^{-\kappa|x-z'|}}{|x-z'|^{l+\frac12}}}\Bigg)^2\mathbb E[\rho(z)\rho(z')]dzdz'\\
&=\frac{\kappa_{\rm r}^{-m}}{8^4|\kappa|^{10}}\int_D\frac{e^{-4\kappa_{\rm i}|x-z|}}{|x-z|^2}\mu(z)dz+O\left(\kappa_{\rm r}^{-m-11}\right)
\end{align*}
for any $x\in U$ and $\kappa_{\rm r}\to\infty$.

Combining the above estimates leads to 
\begin{align}\label{eq:u1_2d}
\mathbb E|u_1(x,k)|^2 &=\mathbb E|u_1^{(1)}(x,k)|^2+2\Re\mathbb E\big[\overline{u_1^{(1)}(x,k)}(u_1(x,k)-u_1^{(1)}(x,k))\big]+\mathbb E\big|u_1(x,k)-u_1^{(1)}(x,k)\big|^2\notag\\
&=\frac{\kappa_{\rm r}^{-m}}{8^4|\kappa|^{10}}\int_D\frac{e^{-4\kappa_{\rm i}|x-z|}}{|x-z|^2}\mu(z)dz+O\big(\kappa_{\rm r}^{-m-11}\big)+O\big((\kappa_{\rm r}^{-m}|\kappa|^{-10})^{\frac12}\kappa_{\rm r}^{-7}\big)+O\big(\kappa_{\rm r}^{-14}\big)\notag\\
&=\frac{\kappa_{\rm r}^{-m}}{8^4|\kappa|^{10}}\int_D\frac{e^{-4\kappa_{\rm i}|x-z|}}{|x-z|^2}\mu(z)dz+O\big(\kappa_{\rm r}^{-m-11}\big)\quad\forall~x\in U.
\end{align}

The following theorem is concerned with the contribution of $u_1$ to the reconstruction formula for both the two- and three-dimensional problems.

\begin{theorem}\label{tm:u1}
Let the random potential $\rho$ satisfy Assumption \ref{as:rho} and $U\subset\mathbb R^d$ be a bounded domain having a positive distance to the support $D$ of $\rho$. For any $x\in U$, it holds
\begin{align}\label{eq:u1_1}
\lim_{K\to\infty}\frac1K\int_K^{2K}\kappa_{\rm r}^{m+14-2d}\mathbb E|u_1(x,k)|^2d\kappa_{\rm r}=T_d(x),
\end{align}
where $T_d(x)$ is given in Theorem \ref{tm:main}. Moreover, if  $\sigma=0$, then it holds 
\begin{align}\label{eq:u1_2}
\lim_{K\to\infty}\frac1{K}\int_{K}^{2K}\kappa^{m+14-2d}|u_1(x,k)|^2d\kappa=T_d(x)\quad\mathbb P\text{-}a.s.
\end{align}
\end{theorem}

\begin{proof}
To prove \eqref{eq:u1_1}, we consider the imaginary part of $\kappa$ as a function of $\kappa_{\rm r}$, i.e., $\kappa_{\rm i}=\kappa_{\rm i}(\kappa_{\rm r})$, which satisfies $\lim_{\kappa_{\rm r}\to\infty}\kappa_{\rm i}(\kappa_{\rm r})=0$. From \eqref{eq:u1_3d} and \eqref{eq:u1_2d}, we get 
\[
\lim_{\kappa_{\rm r}\to\infty}\kappa_{\rm r}^{m+14-2d}\mathbb E|u_1(x,k)|^2=T_d(x).
\]
Based on the mean value theorem, \eqref{eq:u1_1} follows from the identity
\[
\lim_{\kappa_{\rm r}\to\infty}\kappa_{\rm r}^{m+14-2d}\mathbb E|u_1(x,k)|^2=\lim_{K\to\infty}\frac1K\int_K^{2K}\kappa_{\rm r}^{m+14-2d}\mathbb E|u_1(x,k)|^2d\kappa_{\rm r}. 
\]

It then suffices to show \eqref{eq:u1_2} for the case $\sigma=0$, i.e., $\kappa=\kappa_{\rm r}=k^{\frac12}\in\mathbb R_+$. Noting 
\[
\lim_{k\to\infty}e^{-4\kappa_{\rm i}|x-z|}=1,
\]
and combining \eqref{eq:kappari} and \eqref{eq:u1_1}, we have 
\begin{align*}
\lim_{k\to\infty}\kappa^{m+14-2d}\mathbb E|u_1(x,k)|^2=T_d(x). 
\end{align*}

To replace the expectation in the above formula by the frequency average, an asymptotic version of the law of large numbers is required. Such a replacement is an analogue of ergodicity in the frequency domain, and has been adopted in the analysis of stochastic inverse problems (cf. \cite{LPS08,LLW,LW2}).

For $d=3$, we consider the correlations $\mathbb E[u_1(x,k_1)\overline{u_1(x,k_2)}]$ and $\mathbb E[u_1(x,k_1)u_1(x,k_2)]$ with $k_i=\kappa_i^2, i=1,2$ at different wavenumbers $\kappa_1$ and $\kappa_2$. Following the same procedure as that used in \cite[Lemma 4.1]{LW2}, we may show that  
\begin{align*}
\big|\mathbb E[u_1(x,k_1)\overline{u_1(x,k_2)}]\big| &\lesssim \kappa_1^{-4}\kappa_2^{-4}\left[(\kappa_1+\kappa_2)^{-m}(1+|\kappa_1-\kappa_2|)^{-M_1}+\kappa_1^{-M_2}+\kappa_2^{-M_2}\right],\\
\big|\mathbb E[u_1(x,k_1)u_1(x,k_2)]\big| &\lesssim\kappa_1^{-4}\kappa_2^{-4}\left[(\kappa_1+\kappa_2)^{-M_1}(1+|\kappa_1-\kappa_2|)^{-m}+\kappa_1^{-M_2}+\kappa_2^{-M_2}\right],
\end{align*}
where $M_1,M_2>0$ are arbitrary integers. The above estimates indicate the asymptotic independence of $u_1(x,k_1)$ and $u_1(x,k_2)$ for $|\kappa_1-\kappa_2|\gg1$. Then, according to \cite[Theorem 4.2]{LW2}, the expectation in \eqref{eq:u1_1} can be replaced by the frequency average with respect to $\kappa$:
\begin{align*}
\lim_{K\to\infty}\frac1{K}\int_{K}^{2K}\kappa^{m+8}|u_1(x,k)|^2d\kappa=T_3(x)\quad\mathbb P\text{-}a.s.
\end{align*}

For $d=2$, we need to consider $u_1^{(3)}$, which is the truncated $u_1$ with $N=3$. Its correlations at different wavenumbers can be carried out similarly as those for the three-dimensional case (cf. \cite[Lemma 4.4]{LW2}). Hence
\begin{align}\label{eq:u13}
\lim_{K\to\infty}\frac1{K}\int_{K}^{2K}\kappa^{m+10}|u_1^{(3)}(x,k)|^2d\kappa=T_2(x)\quad\mathbb P\text{-}a.s.
\end{align}
The residual $u_1-u_1^{(3)}$ satisfies 
\begin{align*}
|u_1(x,k)-u_1^{(3)}(x,k)|&=\left|\int_D(\Phi^2-\Phi_3^2)(x,z,k)\rho(z)dz\right|\\
&\lesssim\|\Phi^2(x,\cdot,k)-\Phi_3^2(x,\cdot,k)\|_{W^{1,q}(D)}\|\rho\|_{W^{-1,p}(D)}\\
&\lesssim\|\Phi(x,\cdot,k)+\Phi_3(x,\cdot,k)\|_{W^{1,2q}(D)}\|\Phi(x,\cdot,k)-\Phi_3(x,\cdot,k)\|_{W^{1,2q}(D)}\|\rho\|_{W^{-1,p}(D)}\\
&\lesssim k^{-\frac34}\kappa^{-\frac{11}2}\lesssim\kappa^{-7}\quad\mathbb P\text{-}a.s.
\end{align*}
for any $p>1$ and $q$ satisfying $\frac1p+\frac1q=1$,
where we used Lemmas \ref{lm:Phi} and \ref{lm:trun}, and  $\rho\in W^{\frac{m-2}2-\epsilon,p}(D)\subset W^{-1,p}(D)$ for $m\in(1,2]$ and any sufficiently small $\epsilon\in(0,\frac m2)$. We have from a simple calculation that
\begin{align*}
\lim_{K\to\infty}\frac1{K}\int_{K}^{2K}\kappa^{m+10}|u_1(x,k)-u_1^{(3)}(x,k)|^2d\kappa\lesssim\lim_{K\to\infty}\frac1{K}\int_{K}^{2K}\kappa^{m-4}d\kappa=0\quad\mathbb P\text{-}a.s..
\end{align*}
Combining the above estimate with \eqref{eq:u13} leads to
\begin{align*}
\lim_{K\to\infty}\frac1{K}\int_{K}^{2K}\kappa^{m+10}|u_1(x,k)|^2d\kappa=T_2(x)\quad\mathbb P\text{-}a.s.,
\end{align*}
which completes the proof of \eqref{eq:u1_2}.
\end{proof}

\subsection{The analysis of $u_2$} 

It follows from \eqref{borns} and \eqref{u1} that 
\[
u_2(x,k)=\int_{\mathbb R^d}\Phi(x,z,k)\rho(z)u_1(z,x,k)dz=\int_{\mathbb R^d}\int_{\mathbb R^d}\Phi(x,z,k)\rho(z)\Phi(z,z',k)\rho(z')\Phi(z',x,k)dzdz',
\]
which does not contribute to the inversion formula as stated in the following theorem.

\begin{theorem}\label{tm:u2}
Let the random potential $\rho$ satisfy Assumption \ref{as:rho} and $U\subset\mathbb R^d$ be a bounded and convex domain having a positive distance to the support $D$ of $\rho$. For any $x\in U$, it holds 
\[
\lim_{K\to\infty}\frac1{K}\int_K^{2K}\kappa_{\rm r}^{m+14-2d}|u_2(x,k)|^2d\kappa_{\rm r}=0\quad\mathbb P\text{-}a.s.
\]
\end{theorem}

\begin{proof}
The proof is motivated by \cite{LPS08}, where the inverse random potential scattering problem is studied for the two-dimensional Schr\"odinger equation with $m\ge d$. In what follows, we provide some details to demonstrate the differences for the biharmonic wave equation of rougher potentials with $m\in(d-1,d]$.

(i) First we consider the case $d=3$. As a function of $x$ and $\kappa_{\rm r}$, $u_2(x,k)$ satisfies
\begin{align*}
\frac1{K}\int_K^{2K}\kappa_{\rm r}^{m+8}|u_2(x,k)|^2d\kappa_{\rm r}
&\le\int_K^{2K}\frac{\kappa_{\rm r}}{K}\kappa_{\rm r}^{m+7}|u_2(x,k)|^2d\kappa_{\rm r}\\
&\le\int_1^{\infty}\min\left\{2,\frac{\kappa_{\rm r}}{K}\right\}\kappa_{\rm r}^{m+7}|u_2(x,k)|^2d\kappa_{\rm r}\quad\mathbb P\text{-}a.s.
\end{align*}
Then the required result is obtained by taking $K\to\infty$ if the following estimate holds:
\begin{align}\label{eq:int}
\int_1^\infty \kappa_{\rm r}^{m+7}\mathbb E|u_2(x,k)|^2d\kappa_{\rm r}<\infty\quad\forall\,x\in U.
\end{align}

To deal with the product of the rough potentials in $\mathbb E|u_2(x,k)|^2$, we consider the smooth modification $\rho_\varepsilon:=\rho*\varphi_\varepsilon$ with $\varphi_\varepsilon(x)=\varepsilon^{-2}\varphi(x/\varepsilon)$ for $\varepsilon>0$ and $\varphi\in C_0^\infty(\mathbb R^3)$. Define
\begin{align*}
u_{2,\varepsilon}(x,k):&=\int_{\mathbb R^d}\int_{\mathbb R^d}\Phi(x,z,k)\rho_\varepsilon(z)\Phi(z,z',k)\rho_\varepsilon(z')\Phi(z',x,k)dzdz'\\
&=-\frac{1}{(8\pi\kappa^2)^3}\int_{D}\int_{D}\frac{(e^{{\rm i}\kappa|x-z|}-e^{-\kappa|x-z|})e^{{\rm i}\kappa|z-z'|}(e^{{\rm i}\kappa|x-z'|}-e^{-\kappa|x-z'|})}{|x-z||z-z'||x-z'|}\rho_\varepsilon(z)\rho_\varepsilon(z')dzdz'\\
&\quad +\frac{1}{(8\pi\kappa^2)^3}\int_{D}\int_{D}\frac{(e^{{\rm i}\kappa|x-z|}-e^{-\kappa|x-z|})e^{-\kappa|z-z'|}(e^{{\rm i}\kappa|x-z'|}-e^{-\kappa|x-z'|})}{|x-z||z-z'||x-z'|}\rho_\varepsilon(z)\rho_\varepsilon(z')dzdz'\\
&=:-\frac{1}{(8\pi\kappa^2)^3}{\rm II}_1(x,k,\varepsilon)+\frac{1}{(8\pi\kappa^2)^3}{\rm II}_2(x,k,\varepsilon).
\end{align*}
Note that
\begin{align*}
\int_1^\infty \kappa_{\rm r}^{m+7}\mathbb E|u_{2,\varepsilon}(x,k)|^2d\kappa_{\rm r}\lesssim\sum_{i=1}^2\int_1^\infty|\kappa|^{-12}\kappa_{\rm r}^{m+7}\mathbb E|{\rm II}_i(x,k,\varepsilon)|^2d\kappa_{\rm r}\lesssim\sum_{i=1}^2\int_1^\infty\mathbb E|{\rm II}_i(x,k,\varepsilon)|^2d\kappa_{\rm r},
\end{align*}
where in the last inequality we used 
\[
|\kappa|^{-12}\kappa_{\rm r}^{m+7}\le\kappa_{\rm r}^{m-5}\le1\quad \forall\,m\in(2,3].
\]
Based on the Fubini theorem and Fatou's lemma, to show \eqref{eq:int}, it suffices to prove
\begin{align*}
\sup_{\varepsilon\in(0,1)}\int_1^\infty\mathbb E|{\rm II}_i(x,k,\varepsilon)|^2d\kappa_{\rm r}<\infty\quad\forall\,x\in U,~i=1,2.
\end{align*}

The estimates for ${\rm II}_1$ and ${\rm II}_2$ are parallel, and they are similar to the procedure used in \cite{LPS08,LLW} for the inverse potential scattering problems of the two-dimensional acoustic and elastic wave equations without attenuation. The basic idea is to rewrite each term ${\rm II}_i$, $i=1,2$, as the Fourier or inverse Fourier transform of some well-defined function. In the following, we only give the estimate for ${\rm II}_1$ to show the differences in handling the attenuation.

Denote
\[
\mathbb K(x,z,z'):=\frac{(e^{{\rm i}\kappa|x-z|}-e^{-\kappa|x-z|})e^{-{\rm i}\kappa_{\rm r}|x-z|}e^{-\kappa_{\rm i}|z-z'|}e^{-{\rm i}\kappa_{\rm r}|z'-x|}(e^{{\rm i}\kappa|x-z'|}-e^{-\kappa|x-z'|})}{|x-z||z-z'||x-z'|},
\]
then ${\rm II}_1$ can be rewritten as
\begin{align*}
{\rm II}_1(x,k,\varepsilon)=&\int_D\int_De^{{\rm i}\kappa_{\rm r}(|x-z|+|z-z'|+|z'-x|)}\mathbb K(x,z,z')\rho_\varepsilon(z)\rho_\varepsilon(z')dzdz'.
\end{align*}
Define a phase function
\[
L(z,z')=|x-z|+|z-z'|+|z'-x|,
\]
which is uniformly bounded below and above for any $(z,z')\in D\times D$ and $x\in U$. Hence the set
\[
\{(z,z')\in D\times D: L(z,z')=t\},\quad t>0
\]
is non-empty only for $t$ lying in a finite interval $[T_0,T_1]$ with $0<T_0<T_1$.

For any fixed $\tilde t\in[T_0,T_1]$, there exist $\eta=\eta(\tilde t)$ and an open cone $K=K(\tilde t)\subset\mathbb R^6$ such that
\[
D\times D~\cap\{(z,z'): t_0<L(z,z')<t_1\}\subset K\cap\{(z,z'): t_0<L(z,z')<t_1\}=:\Gamma,
\]
where $t_0=\tilde t-\eta$ and $t_1=\tilde t+\eta$. Letting $\Gamma_t:=\Gamma\cap\{(z,z'): L(z,z')=t\}$, we have
\begin{align*}
&\int_\Gamma e^{{\rm i}\kappa_{\rm r}L(z,z')}\mathbb K(x,z,z')\rho_\varepsilon(z)\rho_\varepsilon(z')dzdz'\\
&=\int_{t_0}^{t_1}e^{{\rm i}\kappa_{\rm r}t}\left[\int_{\Gamma_t}\mathbb K(x,z,z')|\nabla L(z,z')|^{-1}\rho_\varepsilon(z)\rho_\varepsilon(z')d\mathcal H^5(z,z')\right]dt\\
&=:\int_{t_0}^{t_1}e^{{\rm i}\kappa_{\rm r}t}S_\varepsilon(t)dt=\mathcal F[S_\varepsilon](-\kappa_{\rm r}),
\end{align*}
where  $\mathcal H^5$ is the Hausdorff measure on $\Gamma_t$ and $S_{\varepsilon}$ is compactly supported in $[T_0,T_1]$. Applying Parseval's identity yields  
\[
\int_1^\infty\mathbb E|{\rm II}_1(x,k,\varepsilon)|^2d\kappa_{\rm r}\lesssim\mathbb E\|S_\varepsilon\|^2_{L^2(T_0,T_1)}.
\]

Using Isserlis' theorem, we obtain 
\begin{align*}
\mathbb E|S_\varepsilon(t)|^2&=\int_{\Gamma_t}\int_{\Gamma_t}\mathbb K(x,z_1,z'_1)\overline{\mathbb K(x,z_2,z'_2)}|\nabla L(z_1,z'_1)|^{-1}|\nabla L(z_2,z'_2)|^{-1}\\
&\quad \times\mathbb E\left[\rho_\varepsilon(z_1)\rho_\varepsilon(z'_1)\rho_\varepsilon(z_2)\rho_\varepsilon(z'_2)\right]d\mathcal H^5(z_1,z'_1)d\mathcal H^5(z_2,z'_2)\\
&=\int_{\Gamma_t}\int_{\Gamma_t}\mathbb K(x,z_1,z'_1)\overline{\mathbb K(x,z_2,z'_2)}|\nabla L(z_1,z'_1)|^{-1}|\nabla L(z_2,z'_2)|^{-1}\\
&\quad \times\Big(\mathbb E[\rho_\varepsilon(z_1)\rho_\varepsilon(z'_1)]\mathbb E[\rho_\varepsilon(z_2)\rho_\varepsilon(z'_2)]+\mathbb E[\rho_\varepsilon(z_1)\rho_\varepsilon(z_2)]\mathbb E[\rho_\varepsilon(z'_1)\rho_\varepsilon(z'_2)]\\
&\quad +\mathbb E[\rho_\varepsilon(z_1)\rho_\varepsilon(z'_2)]\mathbb E[\rho_\varepsilon(z'_1)\rho_\varepsilon(z_2)]\Big)d\mathcal H^5(z_1,z'_1)d\mathcal H^5(z_2,z'_2),
\end{align*}
where $\mathbb K$ and $\nabla L$ satisfy
$|\mathbb K(x,z,z')|\lesssim|z-z'|^{-1}$ and $0<C_1\le|\nabla L(z,z')|\le C_2$, respectively, for any $(z,z')\in D\times D$ with $z\neq z'$ (cf. \cite{LPS08}), and $|\mathbb E[\rho_\varepsilon(z)\rho_\varepsilon(z')]|\lesssim|z-z'|^{m-3-\epsilon}$ for any $\epsilon>0$ and $m\in(2,3]$ according to \eqref{eq:kernel}.
It follows from the H\"older inequality and the symmetry of the integral that
\begin{align*}
\mathbb E|S_\varepsilon(t)|^2&\lesssim\int_{\Gamma_t}\int_{\Gamma_t}|z_1-z'_1|^{-1}|z_2-z'_2|^{-1}|z_1-z'_1|^{m-3-\epsilon}|z_2-z'_2|^{m-3-\epsilon}d\mathcal H^5(z_1,z'_1)d\mathcal H^5(z_2,z'_2)\\
&\quad+\int_{\Gamma_t}\int_{\Gamma_t}|z_1-z'_1|^{-1}|z_2-z'_2|^{-1}|z_1-z_2|^{m-3-\epsilon}|z'_1-z'_2|^{m-3-\epsilon}d\mathcal H^5(z_1,z'_1)d\mathcal H^5(z_2,z'_2)\\
&\quad+\int_{\Gamma_t}\int_{\Gamma_t}|z_1-z'_1|^{-1}|z_2-z'_2|^{-1}|z_1-z'_2|^{m-3-\epsilon}|z'_1-z_2|^{m-3-\epsilon}d\mathcal H^5(z_1,z'_1)d\mathcal H^5(z_2,z'_2)\\
&=\left(\int_{\Gamma_t}|z_1-z'_1|^{m-4-\epsilon}d\mathcal H^5(z_1,z'_1)\right)^2\\
&\quad +2\int_{\Gamma_t}\int_{\Gamma_t}|z_1-z'_1|^{-1}|z_2-z'_2|^{-1}|z_1-z_2|^{m-3-\epsilon}|z'_1-z'_2|^{m-3-\epsilon}d\mathcal H^5(z_1,z'_1)d\mathcal H^5(z_2,z'_2)\\
\lesssim&\left(\int_{\Gamma_t}|z_1-z'_1|^{m-4-\epsilon}d\mathcal H^5(z_1,z'_1)\right)^2\\
&\quad +\left[\int_{\Gamma_t}\int_{\Gamma_t}|z_1-z'_1|^{-3}|z_2-z'_2|^{-3}d\mathcal H^5(z_1,z'_1)d\mathcal H^5(z_2,z'_2)\right]^{\frac13}\\
&\quad \times\left[\int_{\Gamma_t}\int_{\Gamma_t}|z_1-z_2|^{\frac32(m-3-\epsilon)}|z'_1-z'_2|^{\frac32(m-3-\epsilon)}d\mathcal H^5(z_1,z'_1)d\mathcal H^5(z_2,z'_2)\right]^{\frac23}\\
&\lesssim\left(\int_{\Gamma_t}|z_1-z'_1|^{m-4-\epsilon}d\mathcal H^5(z_1,z'_1)\right)^2+\left(\int_{\Gamma_t}|z_1-z'_1|^{-3}d\mathcal H^5(z_1,z'_1)\right)^{\frac43}\\
&\quad +\left(\int_{\Gamma_t}\int_{\Gamma_t}|z_1-z_2|^{3(m-3-\epsilon)}d\mathcal H^5(z_1,z'_1)d\mathcal H^5(z_2,z'_2)\right)^{\frac43},
\end{align*}
where the boundedness of all the last three integrals can be obtained similarly to the two-dimensional problem shown in \cite[Lemma 6]{LPS08}.
 
(ii) Next we consider the case $d=2$. Define the following auxiliary functions (cf. \cite[Section 5.2]{LLW}) via the truncated 
fundamental solution $\Phi_0$: 
\begin{align*}
u_{2,l}(x,k):&=\int_{\mathbb R^d}\int_{\mathbb R^d}\Phi_0(x,z,k)\rho(z)\Phi(z,z',k)\rho(z')\Phi(z',x,k)dzdz',\\
u_{2,r}(x,k):&=\int_{\mathbb R^d}\int_{\mathbb R^d}\Phi_0(x,z,k)\rho(z)\Phi(z,z',k)\rho(z')\Phi_0(z',x,k)dzdz',\\
v(x,k):&=\int_{\mathbb R^d}\int_{\mathbb R^d}\Phi_0(x,z,k)\rho(z)\Phi_0(z,z',k)\rho(z')\Phi_0(z',x,k)dzdz'.
\end{align*}
By Lemmas \ref{lm:Phi}, \ref{lm:operatorK}, and \ref{lm:trun}, we have 
\begin{align*}
|u_2(x,k)-u_{2,l}(x,k)|&\lesssim\|\rho\|_{W^{-\gamma,p}(D)}\left\|\left[\Phi(x,\cdot,k)-\Phi_0(x,\cdot,k)\right]\mathcal K_k\Phi(\cdot,x,k)\right\|_{W^{\gamma,q}(D)}\\
&\lesssim\|\Phi(x,\cdot,k)-\Phi_0(x,\cdot,k)\|_{W^{\gamma,2q}(D)}\|\mathcal K_k\|_{\mathcal L(W^{\gamma,2q}(D))}\|\Phi(\cdot,x,k)\|_{W^{\gamma,2q}(D)}\\
&\lesssim |\kappa|^{-\frac72+\gamma}k^{\gamma-\frac1q-\frac12}k^{-\frac54+\frac\gamma2}\lesssim\kappa_{\rm r}^{-7-\frac2q+4\gamma}\quad\mathbb P\text{-}a.s.,
\end{align*}
\begin{align*}
|u_{2,l}(x,k)-u_{2,r}(x,k)|&\lesssim\|\rho\|_{W^{-\gamma,p}(D)}\left\|\Phi_0(x,\cdot,k)\mathcal K_k\left[\Phi(\cdot,x,k)-\Phi_0(\cdot,x,k)\right]\right\|_{W^{\gamma,q}(D)}\\
&\lesssim\|\Phi_0(x,\cdot,k)\|_{W^{\gamma,2q}(D)}\|\mathcal K_k\|_{\mathcal L(W^{\gamma,2q}(D))}\|\Phi(\cdot,x,k)-\Phi_0(\cdot,x,k)\|_{W^{\gamma,2q}(D)}\\
&\lesssim\kappa_{\rm r}^{-7-\frac2q+4\gamma}\quad\mathbb P\text{-}a.s.,
\end{align*}
\begin{align*}
|u_{2,r}(x,k)-v(x,k)|&\lesssim\|\Phi(\cdot,\cdot,k)-\Phi_0(\cdot,\cdot,k)\|_{W^{\gamma,\tilde q}(D\times D)}\|(\rho\otimes\rho)(\Phi_0\otimes\Phi_0(x,\cdot,k))\|_{W^{-2\gamma,\tilde p}(D\times D)}\\
&\lesssim|\kappa|^{-\frac72+\gamma}\|\rho\|_{W^{-\gamma,\infty}(D)}^2\|\Phi_0(x,\cdot,k)\otimes\Phi_0(\cdot,x,k)\|_{W^{2\gamma,\infty}(D\times D)}\\
&\lesssim\kappa_{\rm r}^{-\frac{17}2+4\gamma}\quad\mathbb P\text{-}a.s.,
\end{align*}
where $(p,q)$ and $(\tilde p,\tilde q)$ are conjugate pairs with $q>1$, $\gamma\in(\frac{2-m}2,\frac12+\frac1q)$, and $\tilde q\in(1,\frac43)$.  
Choosing $q=\frac1{1-\epsilon}$ and $\gamma=\frac{2-m}2+\epsilon$ with a sufficiently small $\epsilon>0$ in above estimates, we get
\begin{align*}
&\lim_{K\to\infty}\frac1K\int_K^{2K}\kappa_{\rm r}^{m+10}|u_2(x,k)-v(x,k)|^2d\kappa_{\rm r}\\
&\lesssim\lim_{K\to\infty}\frac1K\int_K^{2K}\kappa_{\rm r}^{m+10}\Big(\kappa_{\rm r}^{-7-\frac2q+4\gamma}+\kappa_{\rm r}^{-\frac{17}2+4\gamma}\Big)^2d\kappa_{\rm r}\\
&\lesssim\lim_{K\to\infty}\frac1K\int_K^{2K}\left(\kappa_{\rm r}^{-3m+12\epsilon}+\kappa_{\rm r}^{1-3m+8\epsilon}\right)d\kappa_{\rm r}=0\quad\mathbb P\text{-}a.s.
\end{align*}
Hence, to show the result in the theorem, it suffices to prove that the contribution of $v$ is zero. Similar to the three-dimensional case, we consider the smooth modification 
\begin{align*}
v_\varepsilon(x,k):&=\int_{\mathbb R^d}\int_{\mathbb R^d}\Phi_0(x,z,k)\rho_\varepsilon(z)\Phi_0(z,z',k)\rho_\varepsilon(z')\Phi_0(z',x,k)dzdz'\\
&=-\frac{{\rm i}}{8^3\kappa^{\frac{15}2}}\int_D\int_D\frac{({\rm i}e^{{\rm i}\kappa|x-z|}-{\rm i}^{\frac12}e^{-\kappa|x-z|})e^{{\rm i}\kappa|z-z'|}({\rm i}e^{{\rm i}\kappa|z'-x|}-{\rm i}^{\frac12}e^{-\kappa|z'-x|})}{|x-z|^{\frac12}|z-z'|^{\frac12}|z'-x|^{\frac12}}\rho_\varepsilon(z)\rho_\varepsilon(z')dzdz'\\
&\quad +\frac{{\rm i}^{\frac12}}{8^3\kappa^{\frac{15}2}}\int_D\int_D\frac{({\rm i}e^{{\rm i}\kappa|x-z|}-{\rm i}^{\frac12}e^{-\kappa|x-z|})e^{-\kappa|z-z'|}({\rm i}e^{{\rm i}\kappa|z'-x|}-{\rm i}^{\frac12}e^{-\kappa|z'-x|})}{|x-z|^{\frac12}|z-z'|^{\frac12}|z'-x|^{\frac12}}\rho_\varepsilon(z)\rho_\varepsilon(z')dzdz'\\
&=:-\frac{{\rm i}}{8^3\kappa^{\frac{15}2}}{\rm\tilde{II}}_1(x,k,\varepsilon)+\frac{{\rm i}^{\frac12}}{8^3\kappa^{\frac{15}2}}{\rm\tilde{II}}_2(x,k,\varepsilon).
\end{align*}
Following the same procedure as used in the three-dimensional case, we may show 
\begin{align*}\label{eq:intv}
\int_1^\infty\kappa_{\rm r}^{m+9}\mathbb E|v_\varepsilon(x,k)|^2d\kappa_{\rm r}\lesssim\sum_{i=1}^2\int_1^\infty\mathbb E|{\rm\tilde{II}}_{i}(x,k,\varepsilon)|^2d\kappa_{\rm r}<\infty\quad\forall\,x\in U,
\end{align*}
which completes the proof.
\end{proof}

\subsection{The analysis of residual}

Taking out $u_1$ and $u_2$, we define the residual in the Born series
\[
b(x,k):=\sum_{n=3}^\infty u_n(x,k),
\]
which has no contribution to the reconstruction formula as shown in the following theorem. 

\begin{theorem}\label{tm:b}
Let assumptions in Theorem \ref{tm:u2} hold. Then for any $x\in U$, it holds
\[
\lim_{k\to\infty}\kappa_{\rm r}^{m+14-2d}|b(x,k)|^2=0\quad\mathbb P\text{-}a.s.
\]
\end{theorem}

\begin{proof}
Following the similar estimate in \eqref{eq:resi} with $N=2$, we have 
\begin{align*}
\|b(\cdot,k)\|_{L^\infty(U)}\le\sum_{n=3}^\infty\|\mathcal K_k^{n}u_0(\cdot,k)\|_{L^\infty(U)}\lesssim k^{3s+\frac d2-\frac{25}4+\frac{\epsilon}4}\lesssim\kappa_{\rm r}^{6s+d-\frac{25}2+\frac{\epsilon}2}\quad\mathbb P\text{-}a.s.
\end{align*}
for any $s\in(\frac{d-m}2,\frac32)$, $\kappa_{\rm r}\ge C_{k_0}$ and $\epsilon>0$, where $C_{k_0}=\Re[\kappa(k_0)]$ is the a constant depending on $k_0$ given in Lemma \ref{lm:Born}. Hence, we obtain by choosing $s=\frac{d-m}2+\epsilon$ that 
\begin{align*}
\kappa_{\rm r}^{m+14-2d}|b(x,k)|^2
\lesssim\kappa_{\rm r}^{6d-5m-11+13\epsilon}\to0\quad\mathbb P\text{-}a.s.
\end{align*}
as $k\to\infty$ under the condition $m\in(d-1,d]$, which completes the proof.
\end{proof}

\subsection{The proof of Theorem \ref{tm:main}}

Considering the Born series of the scattered field
\[
u^s(x,k)=u_1(x,k)+u_2(x,k)+b(x,k)
\]
for $k\ge k_0$ with $k_0$ being given in Lemma \ref{lm:Born}, we obtain  
\begin{align*}
&\frac1K\int_K^{2K}\kappa_{\rm r}^{m+14-2d}\mathbb E|u^s(x,k)|^2d\kappa_{\rm r}\\
&=\frac1K\int_K^{2K}\kappa_{\rm r}^{m+14-2d}\mathbb E|u_1(x,k)|^2d\kappa_{\rm r}
+\frac1K\int_K^{2K}\kappa_{\rm r}^{m+14-2d}\mathbb E|u_2(x,k)|^2d\kappa_{\rm r}\\
&\quad +\frac1K\int_K^{2K}\kappa_{\rm r}^{m+14-2d}\mathbb E|b(x,k)|^2d\kappa_{\rm r}
+2\Re\left[\frac1K\int_K^{2K}\kappa_{\rm r}^{m+14-2d}\mathbb E\big[u_1(x,k)\overline{u_2(x,k)}\big]d\kappa_{\rm r}\right]\\
&\quad +2\Re\left[\frac1K\int_K^{2K}\kappa_{\rm r}^{m+14-2d}\mathbb E\big[u_1(x,k)\overline{b(x,k)}\big]d\kappa_{\rm r}\right]
+2\Re\left[\frac1K\int_K^{2K}\kappa_{\rm r}^{m+14-2d}\mathbb E\big[u_2(x,k)\overline{b(x,k)}\big]d\kappa_{\rm r}\right]\\
&=: \mathcal I_1+\mathcal I_2+\mathcal I_3+\mathcal I_4+\mathcal I_5+\mathcal I_6,
\end{align*}
where $\mathcal I_4\lesssim\mathcal I_1^{\frac12}\mathcal I_2^{\frac12}$, 
$\mathcal I_5\lesssim\mathcal I_1^{\frac12}\mathcal I_3^{\frac12}$, and $\mathcal I_6\lesssim\mathcal I_2^{\frac12}\mathcal I_3^{\frac12}$.

According to Theorems \ref{tm:u1}, \ref{tm:u2}, and \ref{tm:b}, it is clear to note 
\begin{align*}
\lim_{K\to\infty}\mathcal I_1=T_d(x),\quad
\lim_{K\to\infty}\mathcal I_j=0,\quad j=2,3,
\end{align*}
which lead to
\[
\lim_{K\to\infty}\frac1K\int_K^{2K}\kappa_{\rm r}^{m+14-2d}\mathbb E|u^s(x,k)|^2d\kappa_{\rm r}=T_d(x).
\]
Then \eqref{eq:main} is deduced by utilizing the equivalence between the following limits:
\begin{align*}
&\lim_{K\to\infty}\frac1{K}\int_{K}^{2K}\kappa_{\rm r}^{m+14-2d}\mathbb E|u^s(x,k)|^2d\kappa_{\rm r}\\
&=2\lim_{K\to\infty}\left[\frac1{2K}\int_1^{2K}\kappa_{\rm r}^{m+14-2d}\mathbb E|u^s(x,k)|^2d\kappa_{\rm r}-\frac1{2K}\int_1^K\kappa_{\rm r}^{m+14-2d}\mathbb E|u^s(x,k)|^2d\kappa_{\rm r}\right]\\
&=\lim_{K\to\infty}\frac1{K}\int_1^{K}\kappa_{\rm r}^{m+14-2d}\mathbb E|u^s(x,k)|^2d\kappa_{\rm r}.
\end{align*}

If $\sigma=0$, then $\kappa=\kappa_{\rm r}=k^{\frac12}$. The expectation in the above estimates can be removed due to Theorem \ref{tm:u1}. We then get
\begin{align*}
T_d(x)&=\lim_{K\to\infty}\frac1{K}\int_1^{K}\kappa^{m+14-2d}|u^s(x,k)|^2d\kappa\\
&=\lim_{K\to\infty}\frac1{K}\int_1^{K^2}k^{\frac{m+14-2d}2}|u^s(x,k)|^2\frac12k^{-\frac12}dk\\
&=\lim_{K\to\infty}\frac1{2K}\int_1^{K^2}k^{\frac{m+13}2-d}|u^s(x,k)|^2dk\quad\mathbb P\text{-}a.s.,
\end{align*}
which completes the proof of \eqref{eq:main2}.

The uniqueness can be proved by following the same argument in \cite[Theorem 1]{LPS08} or \cite[Theorem 4.4]{LW}. 

\section{Conclusion}\label{sec:con}

In this paper, we have studied the random potential scattering for biharmonic waves in lossy media. The unique continuation principle is proved for the biharmonic wave equation with rough potentials. Based on the equivalent Lippmann--Schwinger integral equation, the well-posedness is established for the direct scattering problem in the distribution sense. The uniqueness is attained for the inverse scattering problem. Particularly, we show that the correlation strength of the random potential is uniquely determined by the high frequency limit of the second moment of the scattered wave field averaged over the frequency band. Moreover, we demonstrate that the expectation can be removed and the data of only a single realization is needed almost surely to ensure the uniqueness of the inverse problem when the medium is lossless.

Finally, we point out some important future directions along the line of this research. In this work, the convergence of the Born series is crucial for the inverse problem. However, this approach is not applicable to the inverse random medium scattering problems, since the Born series for the medium scattering problem does not converge any more in the high frequency regime. It is unclear whether the correlation strength of the random medium can be uniquely determined by some statistics of the wave field. Other interesting problems include the inverse random source or potential problems for the wave equations with higher order differential operators, such as the stochastic polyharmonic wave equation.

\end{document}